\documentclass[11pt,reqno]{amsart}
\usepackage[english]{babel}
\usepackage[T1]{fontenc}
\usepackage[utf8]{inputenc}

\usepackage[left=1in,right=1in,top=1in,bottom=1in]{geometry}

\usepackage{float}
\usepackage{tikz}

\usepackage{lineno}
\usepackage{amsfonts}
\usepackage{mathtools}
\usepackage{verbatim}
\usepackage{mathtools}
\usepackage{dsfont}
\usepackage{amsthm}
\usepackage{indentfirst}
\usepackage{amsmath,amssymb}
\usepackage{verbatim}
\usepackage{array}
\usepackage{fancyhdr}
\usepackage{xcolor}
\usepackage{noindentafter}
\usepackage{esint}
\usepackage{mathrsfs}
\usepackage{stmaryrd}
\usepackage{environ}
\usepackage[hidelinks]{hyperref}
\usepackage[toc,page]{appendix}

\usepackage{enumerate}
\usepackage{thmtools}

\NoIndentAfterCmd{\]}
\NoIndentAfterEnv{equation}
\NoIndentAfterEnv{align*}
\NoIndentAfterEnv{itemize}
\NoIndentAfterEnv{enumerate}

\DeclareUnicodeCharacter{00A0}{ }
\DeclareUnicodeCharacter{00A0}{~}

\makeatletter
\@addtoreset{footnote}{section}
\makeatother

\DeclareFontFamily{U}{matha}{\hyphenchar\font45}
\DeclareFontShape{U}{matha}{m}{n}{
  <-6> matha5 <6-7> matha6 <7-8> matha7
  <8-9> matha8 <9-10> matha9
  <10-12> matha10 <12-> matha12
  }{}

\DeclareSymbolFont{matha}{U}{matha}{m}{n}
\DeclareMathSymbol{\Lt}{3}{matha}{"CE}

\DeclareMathOperator{\id}{id}

\newcommand{\R}{\mathbb{R}}

\newcommand{\A}{\mathcal{A}}
\newcommand\res{\mathop{\hbox{\vrule height 7pt width .3pt depth 0pt
\vrule height .3pt width 5pt depth 0pt}}\nolimits}

\DeclareMathOperator{\diag}{diag}
\DeclareMathOperator{\curl}{curl}

\DeclareMathOperator{\cof}{cof}
\DeclareMathOperator{\rank}{rank}
\DeclareMathOperator{\Id}{Id}
\DeclareMathOperator{\di}{d}
\DeclareMathOperator{\dx}{dx}

\DeclareMathOperator{\dv}{div}
\DeclareMathOperator{\tr}{tr}

\DeclareMathOperator{\spann}{span}

\DeclareMathOperator{\Lip}{Lip}
\DeclareMathOperator{\im}{Im}

\DeclareMathOperator{\dist}{dist}

\theoremstyle{plain}

\newtheorem{Teo}{Theorem}[section]
\newtheorem{lemma}[Teo]{Lemma}
\newtheorem{prop}[Teo]{Proposition}

\newtheorem{Cor}[Teo]{Corollary}
\newtheorem{Q}{Question}
\theoremstyle{definition}
\newtheorem{Def}[Teo]{Definition}
\theoremstyle{remark}
\newtheorem{rem}[Teo]{Remark}

\NoIndentAfterEnv{Def}
\NoIndentAfterEnv{lemma}
\NoIndentAfterEnv{prop}
\NoIndentAfterEnv{Cor}
\NoIndentAfterEnv{prop}
\NoIndentAfterEnv{rem}
\NoIndentAfterEnv{Rem}
\NoIndentAfterEnv{Teo}
\NoIndentAfterEnv{Theorem}
\NoIndentAfterEnv{DT}

\title{Geometric measure theory and differential inclusions}

\author[De Lellis]{C. De Lellis}
\address[Camillo De Lellis]{
\newline \indent School of Mathematics, Institute for Advanced Study and Universit\"at Z\"urich 
\newline \indent 1 Einstein Dr., Princeton NJ 05840, USA}
\email{camillo.delellis@math.ias.edu}

\author[De Philippis]{G. De Philippis}
\address[Guido De Philippis]{
\newline \indent SISSA
\newline \indent Via Bonomea 265, 34136, Trieste, Italy}
\email{guido.dephilippis@sissa.it }

\author[Kirchheim]{B. Kirchheim}
\address[Bernd Kirchheim]{
\newline \indent Mathematisches Institut, Universit\"at Leipzig,
\newline \indent Augustusplatz 10, 04109 Leipzig, Germany}
\email{kirchheim@math.uni-leipzig.de}

\author[Tione]{R. Tione}
\address[Riccardo Tione]{
\newline \indent Institut f\"ur Mathematik, Universit\"at Z\"urich,
\newline \indent Winterthurerstrasse 190, CH-8057 Zurich, Switzerland}
\email{riccardo.tione@math.uzh.ch}

\begin{document}

\maketitle

\begin{abstract}
In this paper we consider Lipschitz graphs of functions which are stationary points of strictly polyconvex energies. Such graphs can be thought as integral currents, resp. varifolds, which are stationary for some elliptic integrands. The regularity theory for the latter is a widely open problem, in particular no counterpart of the classical Allard's theorem is known. We address the issue from the point of view of differential inclusions and we show that the relevant ones do not contain the class of laminates which are used in \cite{SMVS} and \cite{LSP} to construct nonregular solutions. Our result is thus an indication that an Allard's type result might be valid for general elliptic integrands. We conclude the paper by listing a series of open questions concerning the regularity of stationary points for elliptic integrands.
\end{abstract}

\section{Introduction}

Let $\Omega\subset\R^m$ be open and $f\in{C}^1(\R^{n\times m},\mathbb{R})$ be a (strictly) polyconvex function, i.e. such that there is a (strictly) convex $g\in C^1$ such that $f (X) = g (\Phi (X))$, where $\Phi (X)$ denotes the vector of subdeterminants of $X$ of all orders. We then consider the following \emph{energy} $\mathds{E}:\Lip(\Omega,\R^n) \to \mathbb{R}$:
\begin{equation}\label{e:energy}
\mathds{E}(u) := \int_\Omega f(D u) dx\, .
\end{equation}

\begin{Def}\label{d:map_stationary}
Consider a map $\bar{u} \in \Lip(\Omega,\R^n)$. The one-parameter family of functions $\bar{u}+\varepsilon v$ will be called {\em outer variations} and $\bar{u}$ will be called \emph{critical for $\mathds{E}$} if
\[
\left.\frac{d}{d\varepsilon}\right|_{\varepsilon = 0}\mathds{E}(\bar{u} + \varepsilon v) = 0 \qquad\qquad\qquad \forall v \in C^\infty_c(\Omega,\mathbb{R}^n) \, .
\]
Given a vector field $\Phi\in C^1_c (\Omega, \R^m)$ we let $X_\varepsilon$ be its flow\footnote{Namely $X_\varepsilon(x) = \gamma_x(\varepsilon)$, where $\gamma_x$ is the solution of the ODE
$\gamma'(t) = \Phi(\gamma(t))$ subject to the initial condition $\gamma(0) = x$.}. The one-parameter family of functions $u_\varepsilon =\bar{u} \circ X_\varepsilon$ will be called an {\em inner variation}. A critical point
$\bar{u} \in \Lip(\Omega,\R^n)$ is \emph{stationary} for $\mathds{E}$ if
\[
\left.\frac{d}{d\varepsilon}\right|_{\varepsilon = 0} \mathds{E}(u_\varepsilon) = 0
\qquad\qquad \forall \Phi\in C^1_c(\Omega,\R^m)\, .
\]
\end{Def}
Classical computations reduce the two conditions above to\footnote{$\langle A,B \rangle := \tr(A^TB)$ denotes the usual Hilbert-Schmidt scalar product of the matrices $A$ and $B$.}, respectively,
\begin{equation}\label{ver}
\int_{\Omega}\langle Df(D \bar{u}),D v\rangle\dx = 0\qquad\qquad \forall v\in C^1_c(\Omega,\mathbb{R}^n).
\end{equation}
and
\begin{equation}\label{hor}
\int_{\Omega}\langle Df(D \bar{u}), D \bar{u} D \Phi\rangle \dx - \int_{\Omega}f(D \bar{u})\dv\, \Phi \dx = 0\qquad\qquad\forall \Phi\in C^1_c(\Omega,\mathbb{R}^m)\, .
\end{equation}
The graphs of Lipschitz functions can be naturally given the structure of integer rectifiable currents (without boundary in $\Omega\times \mathbb R^m$) and of integral varifold, cf. \cite{FED,SIM,MGS1}. In particular the graph of any stationary point $\bar{u}\in \Lip (\Omega, \R^n)$ for a polyconvex energy $\mathds{E}$ can be thought as a stationary point for a corresponding elliptic energy, in the space of integer rectifiable currents and in that of integral varifolds, respectively, see \cite[Chapter 1, Section 2]{MGS2}. Even though this fact is probably well known, it is not entirely trivial and we have not been able to find a reference in the literature: we therefore give a proof for the reader's convenience. Note that a particular example of polyconvex energy is given by the area integrand
\begin{equation}\label{e:area}
\mathcal{A} (X) = \sqrt{\det ({\rm id}_{\mathbb R^{m\times m}} + X^T X)}\, .
\end{equation}
The latter is {\em strongly polyconvex} when restricted to the any ball $B_R \subset \mathbb R^{n\times m}$, namely there is a positive constant $\varepsilon (R)$ such that $X\mapsto \mathcal{A} (X) - \varepsilon (R) |X|^2$ is still polyconvex on $B_R$.

\medskip

When $n=1$ strong polyconvexity reduces to locally uniform convexity and any Lipschitz critical point is therefore $C^{1,\alpha}$ by the De Giorgi-Nash theorem. The same regularity statement holds in the much simpler ``dual case'' $m=1$, where criticality implies that the vector valued map $\bar{u}$ satisfies an appropriate system of ODEs. Remarkably, L. Sz{\'{e}}kelyhidi in \cite{LSP} proved the existence of smooth strongly polyconvex integrands $f: \mathbb R^{2\times 2}\to \R$ for which the corresponding energy has Lipschitz critical points which are nowhere $C^1$. The paper \cite{LSP} is indeed an extension of  a previous groundbreaking result of S. M\"uller and V. \v{S}ver\'ak \cite{SMVS}, where the authors constructed a Lipschitz critical point to a smooth strongly quasiconvex energy (cf. \cite{SMVS} for the relevant definition) which is nowhere $C^1$. A precursor of such examples can be found in the pioneering PhD thesis of V. Scheffer,  \cite{Schef}.

Minimizers of strongly quasiconvex functions have been instead proved to be regular almost everywhere, see \cite{EVA, KRI} and \cite{Schef}. Note that the ``geometric'' counterpart of the latter statement is Almgren's celebrated regularity theorem for integral currents minimizing strongly elliptic integrands~\cite{Almgren68}. Stationary points need not to be local minimizers for the energy, even though every minimizer for an energy is a stationary point. Moreover, combining the uniqueness result in \cite{TAH} and \cite[Theorem 4.1]{SMVS}, it is easy to see that there exist critical points that are not stationary. 

\medskip

Other than the result in \cite{TAH}, not much is known about the properties of stationary points, in particular it is not known whether they must be $C^1$ on a set of full measure. Observe that Allard's $\varepsilon$ regularity theorem applies when $f$ is the area integrand and allows to answer positively to the latter question for $f$ in \eqref{e:area}. The validity of an Allard-type $\varepsilon$-regularity theorem for general elliptic energies is however widely open. 
A first interesting question is whether one could extend the examples of M\"uller and \v{S}ver\`ak and Sz\'ekelyhidi to provide counterexamples. Both in \cite{SMVS} and \cite{LSP}, the starting point of the construction of irregular solutions is rewriting the condition \eqref{ver} as a diff functions erential inclusion, and then finding a so-called $T_N$-configuration ($N = 4$ in the first case, $N = 5$ in the latter) in the set defining the differential inclusion. The main result of the present paper shows that such strategy fails in the case of stationary points. More precisely:

\medskip

\noindent (a) We show that $\bar{u}$ solves \eqref{ver}, \eqref{hor} if and only if there exists an $L^\infty$ matrix field $A$ that solves a certain system of linear, constant coefficients, PDEs and takes almost everywhere values in a fixed set of matrices, which we denote by $K_f$ and call {\em inclusion set}, cf. Lemma \ref{equiv}. The latter system of PDEs will be called a {\em div-curl} differential inclusion, in order to distinguish them from classical
differential inclusions, which are PDE of type $Du\in K$ a.e., and from ``divergence differential inclusions'' as for instance considered in \cite{EEDI}.

\medskip

\noindent (b) We give the appropriate generalization of $T_N$ configurations for {\em div - curl} differential inclusions, which we will call $T_N'$ configurations, cf. Definition \ref{Def_TN'}. As in the ``classical'' case the latter are subsets of cardinality $N$ of the set $K_f$ which satisfy a particular set of conditions.

\medskip
 
\noindent (c) We then prove the following nonexistence result.

\begin{Teo}\label{t:main}
If $f\in C^1 (\R^{n\times m} )$ is strictly polyconvex, then $K_f$ does not contain any set $\{A_1, \ldots , A_N\}$ which induces a $T_N'$ configuration.
\end{Teo}

\begin{rem}[Sz\'ekelyhidi's result]\label{r:Laszlo} Theorem \ref{t:main} can be directly compared with the results in \cite{LSP}, which concern the ``classical'' differential inclusions induced by \eqref{ver} alone. In particular \cite[Theorem 1]{LSP} shows the existence of a smooth strongly polyconvex integrand $f\in C^\infty (\mathbb R^{2\times 2})$ for which the corresponding ``classical'' differential inclusion contains a $T_5$ configuration (cf. Definition \ref{Def_TN}). In fact the careful reader will notice that the $5$ matrices given in \cite[Example 1]{LSP} are incorrect. This is due to an innocuous sign error of the author in copying their entries. While other $T_5$ configurations can be however easily computed following the approach given in \cite{LSP}, according to \cite{Laszlo_private}, the correct original ones of \cite[Example 1]{LSP} are the following:
\[Z_1\doteq
\left(
\begin{array}{cc}
2&2\\
-2&-2\\
-20&-20\\
-14&-14
\end{array}
\right),
Z_2\doteq
\left(
\begin{array}{cc}
3&5\\
-5&-9\\
0&10\\
-3&1
\end{array}
\right),
Z_3\doteq
\left(
\begin{array}{cc}
4&3\\
-9&-5\\
41&0\\
21&-3
\end{array}
\right),
\]
\[Z_4\doteq
\left(
\begin{array}{cc}
-3&-3\\
8&9\\
-54&-72\\
-30&-41
\end{array}
\right),
Z_5\doteq
\left(
\begin{array}{cc}
0&0\\
-1&-2\\
18&36\\
11&22
\end{array}
\right).
\]

These five matrices form a $T_5$ configuration with $k_i = 2,\forall 1\le i \le 5$, $P = 0$, and rank-one "arms" given by

\[C_1\doteq
\left(
\begin{array}{cc}
1&1\\
-1&-1\\
-10&-10\\
-7&-7
\end{array}
\right),
C_2\doteq
\left(
\begin{array}{cc}
1&2\\
-2&-4\\
5&10\\
2&4
\end{array}
\right),
C_3\doteq
\left(
\begin{array}{cc}
1&0\\
-3&0\\
23&0\\
13&0
\end{array}
\right),
\]
\[C_4\doteq
\left(
\begin{array}{cc}
-3&-3\\
7&7\\
-36&-36\\
-19&-19
\end{array}
\right),
C_5\doteq
\left(
\begin{array}{cc}
0&0\\
-1&-2\\
18&36\\
11&22
\end{array}
\right).
\]

\end{rem}

\medskip

Even though it seems still early to conjecture the validity of partial regularity for stationary points, our result leans toward a positive conclusion: Theorem \ref{t:main} can be thought as a first step in that direction. 

Another indication that an Allard type $\varepsilon$-regularity theorem might be valid for at least some class of energies is provided by the recent paper \cite{DDG} of A. De Rosa, the second named author and F. Ghiraldin, which generalizes Allard's rectifiability theorem  to stationary varifolds of a wide class of energies . In fact the authors' theorem characterizes in terms of an appropriate condition on the integrand (called ``atomic condition'', cf. \cite[Definition 1.1]{DDG}) those energies for which rectifiability of stationary points hold.  Furthermore one can use the ideas in \cite{DDG} to show that the atomic conditions implies strong \(W^{1,p}\) convergence of  sequences of stationary equi-Lipschitz graphs, \cite{AT}. When transported to stationary Lipschitz graphs, the latter is yet another obstruction to applying the methods of \cite{SMVS} and \cite{LSP}. In~\cite{De-RosaKolasinski18} it has been  shown that the atomic condition implies Almgren's ellipticity.  It is an intriguing issue to understand if this implication can be reversed and (if not) to understand wether this (hence stronger) assumption on the integrand can be helpful in establishing regularity of stationary points.

We believe that the connection between differential inclusions and geometric measure theory might be fruitful 
and poses a number of interesting and challenging questions. We therefore conclude this work with some related problems in Section \ref{questions}.

\medskip

The rest of the paper is organized as follows: in Section \ref{DI} we rewrite the Euler Lagrange conditions \eqref{ver} and \eqref{hor} as a div-curl differential inclusion and we determine its wave cone. We then introduce the inclusion set $K_f$ and, after recalling the definition of $T_N$ configurations for classical differential inclusions, we define corresponding $T_N'$ configurations for div-curl differential inclusions. In Section \ref{TC} we give a small extension of a key result of \cite{LSR} on classical $T_N$ configurations. In Section \ref{PF} we consider arbitrary sets of $N$ matrices and give an algebraic characterization of those sets which belong to an inclusion set $K_f$ for some strictly polyconvex $f$. In Section \ref{MT} we then prove the main theorem of the paper, Theorem \ref{t:main}. As already mentioned, Section \ref{s:gmt} discusses the link between stationary graphs and stationary varifolds, whereas Section \ref{questions} is a collection of  open questions.


\section{Div-curl differential inclusions, wave cones and inclusion sets}\label{DI}

As written in the introduction, the Euler-Lagrange conditions for energies $\mathds{E}$ are given by:

\begin{equation}\label{vargr}
\left\{
\begin{array}{ll}
\displaystyle \int_{\Omega}\langle Df(D u),D v\rangle\dx = 0 &\forall v\in C^1_c(\Omega,\mathbb{R}^n) \vspace{1mm}\\ 
\displaystyle \int_{\Omega}\langle Df(D u), D u D \Phi\rangle \dx - \int_{\Omega}f(D u)\dv \Phi \dx = 0\qquad & \forall \Phi\in C_c^1(\Omega,\mathbb{R}^m),
\end{array}\right.
\end{equation}
Here we rewrite the system \eqref{vargr} as a differential inclusion. To do so, it is sufficient to notice that the left hand side of the second equation can be rewritten as
\begin{align*}
\int_{\Omega}\langle Df(D u), D u D \Phi\rangle \dx - \int_{\Omega}f(D u)\dv \Phi \dx
& =\int_{\Omega}\langle D u^T Df(D u), D\Phi \rangle - \langle f(D u)\id,D g\rangle \dx\\
& =\int_{\Omega}\langle D u^T Df(D u) - f(D u)\id, D \Phi\rangle \dx
\end{align*}
Hence, the inner variation equation is the weak formulation of
\[
\dv( D u^T Df(D u) - f(D u)\id) = 0.
\]

Since also the outer variation is the weak formulation of a PDE in divergence form, namely
\[
\dv( Df(D u)) = 0,
\]
we introduce the following terminology:

\begin{Def}
A {\em div-curl differential inclusion} is the following system of partial diffential equations for a triple
of maps $X, Y\in L^\infty (\Omega, \mathbb{R}^{n\times m})$ and $Z\in L^\infty (\Omega, \R^{m\times m})$:
\begin{equation}\label{e:div_curl_free}
{\rm curl}\, X = 0, \qquad {\rm div}\, Y =0, \qquad {\rm div}\, Z = 0\, ,
\end{equation}
\begin{equation}\label{e:inclusion}
W := \left( 
\begin{array}{c}
X\\
Y\\
Z
\end{array}
\right) 
\in K_f := \left\{A \in \R^{(2n + m)\times m}:
A =
\left(
\begin{array}{c}
X\\
Df(X)\\
X^TDf(X) - f(X)\id
\end{array}
\right)
\right\},
\end{equation}
where $f\in C^1 (\R^{n\times m})$ is a fixed function. The subset $K_f \subset \R^{(2n+m)\times m}$ will be called the {\em inclusion set} relative to $f$. 
\end{Def}

The following lemma is then an obvious consequence of the above discussion

\begin{lemma}\label{equiv}
Let $f\in C^1 (\R^{n\times m})$. A map
$u \in \Lip(\Omega,\mathbb{R}^n)$ is a stationary point of the energy \eqref{e:energy} if and only there are matrix fields $Y\in L^\infty (\Omega, \R^{n\times m})$ and $Z\in L^\infty (\Omega, \R^{m\times m})$ such that $W = (Du, Y,Z)$ solves the div-curl differential inclusion \eqref{e:div_curl_free}-\eqref{e:inclusion}. 
\end{lemma}

\subsection{Wave cone for div-curl differential inclusions} We recall here the definition of wave cone for a system of linear constant coefficient first order PDEs.
Given a system of linear constant coefficients PDEs
\begin{equation}\label{e:linearpde}
\sum_{i=1}^m A_i\partial_iz=0
\end{equation}
in the unknown $z: \R^m \supset \Omega \to \R^d$
we consider 
{\it plane wave} solutions to \eqref{e:linearpde}, that is, solutions of the form
\begin{equation}\label{planewave}
z(x)=ah(x\cdot\xi),
\end{equation}
where $h:\R\to\R$. The 
{\it wave cone} $\Lambda$ is given by the states $a\in\R^d$ for which there is a $\xi\neq 0$ such that for any choice of the profile $h$
the function \eqref{planewave} solves \eqref{e:linearpde}, that is,
\begin{equation}\label{e:wavecone}
\Lambda:=\left\{a\in\R^d:\,\exists\xi\in\R^m\setminus\{0\}\quad
\mbox{with}\quad \sum_{i=1}^m\xi_i A_ia=0\right\}.
\end{equation}


The following lemma is then an obvious consequence of the definition and its proof is left to the reader. 

\begin{lemma}\label{l:wavecones}
The wave cone of the system $\curl X =0$ is given by rank one matrices, whereas the wave cone for the system \eqref{e:div_curl_free} is given by triple of matrices $(X, Y, Z)$ for which there is a unit vector $\xi \in \mathbb S^{m-1}$ and a vector $u\in \R^n$ such that $X = u\otimes \xi$, $Y \xi =0$ and $Z \xi =0$. 
\end{lemma}

Motivated by the above lemma we then define 

\begin{Def}\label{d:cone_dc} The cone $\Lambda_{dc}\subset \mathbb R^{(2n+m)\times m}$ consists of the matrices in block form
\[
\left(
\begin{array}{l}
X\\
Y\\
Z
\end{array}\right)
\]
with the property that there is a direction $\xi\in \mathbb S^{m-1}$ and a vector $u\in \mathbb R^n$ such that $X = u\otimes \xi$, $Y \xi =0$ and $Z\xi =0$. 
\end{Def}

\subsection{$T_N$ configurations} We start definining $T_N$ configurations for ``classical'' differential inclusions.

\begin{Def}\label{Def_TN} An ordered set of $N\geq 2$ matrices $\{X_i\}_{i=1}^N \subset \R^{n\times m}$ of distinct matrices is said to \emph{induce a $T_N$ configuration} if there exist matrices $P, C_i \in\R^{n\times m}$ and real numbers $k_i > 1$ such that:
\begin{itemize}
\item[(a)] Each $C_i$ belongs to the wave cone of ${\rm curl}\, X=0$, namely $\rank (C_i) \leq 1$ for each $i$;
\item[(b)] $\sum_i C_i = 0$;
\item[(c)] $X_1, \ldots, X_n$, $P$ and $C_1, \ldots, C_N$ satisfy the following $N$ linear conditions 
\begin{equation}\label{form}
\begin{split}
&X_1 = P + k_1 C_1 ,\\
&X_2 = P + C_1 + k_2C_2 ,\\
&\dots\\
&\dots\\
&X_N = P + C_1 +\dots + k_NC_N\, .
\end{split}
\end{equation}
\end{itemize}
In the rest of the note we will use the word $T_N$ configuration for the data $P, C_1, \ldots , C_N, k_1, \ldots k_N$. 
We will moreover say that the configuration is {\em nondegenerate} if $\rank (C_i)=1$ for every $i$. 
\end{Def}

Note that our definition is more general that the one usually given in the literature (cf. \cite{SMVS,LSP,LSR}) because we drop the requirement that there are no rank one connections between distinct $X_i$ and $X_j$. Moreover, rather than calling $\{X_1, \ldots , X_N\}$ a $T_N$ configuration, we prefer to say that it ``induces'' a $T_N$ configuration, namely we regard the whole data $X_1, \ldots , X_N, C_1, \ldots , C_N, k_1, \ldots , k_N$ since it is not at all clear that given an ordered set $\{X_1,  \ldots , X_N\}$ of distinct matrices there is at most one choice of the matrices $C_1, \ldots , C_N$ and of the coefficients $k_1, \ldots , k_N$ satisfying the conditions above (if we drop the condition that the set is ordered, then it is known that there is more than one choice, see \cite{FS}).  

We observe that the definition of $T_N$ configuration could be split into two parts. A ``geometric part'', namely the points (b) and (c), can be considered as characterizing a certain ``arrangement of $2N$ points'' in the space of matrices, consisting of:
\begin{itemize}
\item A closed piecewise linear loop, loosely speaking a polygon (not necessarily {\em planar}) with vertices  $P_1= P+C_1, P_2 = P+C_1+C_2, \ldots , P_N = P+ C_1+\ldots + C_N = P$;
\item $N$ additional ``arms'' which extend the sides of the polygon, ending in the points $X_1, \ldots , X_N$.
\end{itemize}
See Figure \ref{f:T5} for a graphical illustration of these facts in the case $N=4$.

\begin{figure}
\begin{tikzpicture}
\draw (-1,-1) -- (3.5,-0.4375);
\draw (1,-0.75) -- (0.58, 3.45);
\draw (0.75, 1.75) -- (-2.85, 0.55);
\draw (-1.5, 1)  -- (-0.75, -2);
\draw (-1.2,-0.8) node [anchor = east] {$P+C_1+C_2$};
\draw (-0.75, -2) node [anchor = north] {$X_2$};
\draw (0.75, 1.75) node [anchor = west] {$P$};
\draw (0.58, 3.45) node [anchor = south] {$X_4$};
\draw (1.4, -0.9) node [anchor = north] {$P+C_1+C_2+C_3$};
\draw (3.5, -0.4375) node [anchor = west] {$X_3$};
\draw (-2.85, 0.55) node [anchor = east] {$X_1$};
\draw (-1.7, 1.1) node [anchor = south] {$P+C_1$};
\end{tikzpicture}
\caption{The geometric arrangement of a $T_4$ configuration.}\label{f:T5}
\end{figure}
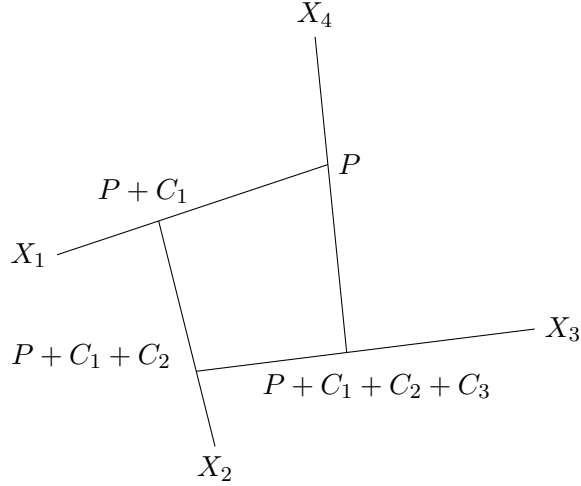

The closing condition in Definition \ref{Def_TN}(b) is a necessary and sufficient condition for the polygonal line to ``close''. Condition (c) determines that each $X_i$ is a point on the line containing the segment $P_{i-1} P_i$. Note that the inequality $k_i >1$ ensures that $X_i$ is external to the segment, ``on the side of $P_i$''. The ``nondegeneracy'' condition is equivalent to the vertices of the polygon being all distinct. Note moreover that, in view of our definition, we include the possibility $N=2$. In the latter case the $T_2$ configuration consists of a single rank one line and of $4$ points $X_1, X_2, C_1, C_2$ lying on it. We have decided to follow this convention, even though this is an unusual choice compared to the literature. 

The second part of the Definition, namely condition (a), is of algebraic nature and related to the fact that $T_N$ configurations are used to study ``classical differential inclusions'', namely PDEs of the form ${\rm curl}\, X =0$. The condition prescribes simply that each vector $X_i-P_i$ belongs to the wave cone of ${\rm curl}\, X=0$. 

\subsection{$T'_N$ configurations}\label{TpN}

In this section we generalize the notion of $T_N$ configuration to div-curl differential inclusions. 
The geometric arrangement remains the same, while the wave cone condition is replaced by the one dictated by the new PDE \eqref{e:div_curl_free}.

\begin{Def}\label{Def_TN'}
A family $\{A_1, \ldots, A_N\}\subset \R^{(2n+m)\times m}$ of $N\geq 2$ {\em distinct}
\[
A_i:=\left(
\begin{array}{c}
X_i\\
Y_i\\
Z_i
\end{array}
\right)
\]
induces a {\em $T_N'$ configuration} if there are matrices $P, Q,  C_i, D_i \in \R^{n\times m}$, $R, E_i\in \R^{m\times m}$ and coefficients $k_i >1$ such that
\[
\left(
\begin{array}{c}
X_i\\
Y_i\\
Z_i
\end{array}
\right) 
= \left(
\begin{array}{c}
P\\
Q\\
R
\end{array}
\right)
+
 \left(
\begin{array}{c}
C_1\\
D_1\\
E_1
\end{array}
\right)
+ \cdots 
+
\left(
\begin{array}{c}
C_{i-1}\\
D_{i-1}\\
E_{i-1}
\end{array}
\right)
+
k_i
\left(
\begin{array}{c}
C_i\\
D_i\\
E_i
\end{array}
\right)
\] 
and the following properties hold:
\begin{itemize}
\item[(a)] each element $(C_i, D_i, E_i)$ belongs to the wave cone $\Lambda_{dc}$ of \eqref{e:div_curl_free}; 
\item[(b)] $\sum_\ell C_\ell = 0$, $\sum_\ell D_\ell =0 $ and $\sum_\ell E_\ell = 0$.
\end{itemize}
We say that the $T'_N$ configuration is {\em nondegeneate} if $\rank (C_i)=1$ for every $i$.
\end{Def}

We collect here some simple consequences of the definition above and of the discussion on $T_N$ configurations. 

\begin{prop}\label{p:T_N'_easy}
Assume $A_1, \ldots , A_N$ induce a $T_N'$ configuration with $P,Q, R, C_i, D_i, E_i$ and $k_i$ as in Definition \ref{Def_TN'}. Then:
\begin{itemize}
\item[(i)] $\{X_1, \ldots , X_N\}$ induce a $T_N$ configuration of the form \eqref{form}, if they are distinct; moreover the $T_N'$ configuration is nondegenerate if and only if the $T_N$ configuration induced by $\{X_1, \ldots , X_N\}$ is nondegenerate;
\item[(ii)] For each $i$ there is an $n_i\in \mathbb S^{m-1}$ and a $u_i\in \R^n$ such that
$C_i = u_i\otimes n_i$, $D_i n_i =0$ and $E_i n_i =0$;
\item[(iii)] $\tr C_i^T D_i = \langle C_i, D_i\rangle = 0$ for every $i$.
\end{itemize}
\end{prop}
\begin{proof}
(i) and (ii) are an obvious consequence of Definition \ref{Def_TN'} and of Definition \ref{d:cone_dc}. After extending $n_i$ to an orthonormal basis $\{n_i, v_2^j, \ldots v_m^j\}$ of $\R^m$ we can explicitely compute
\[
\langle C_i, D_i \rangle = (D_i n_i, C_i n_i) + \sum_{j=2}^m (D_i v_i^j, C_i v_i^j)=0\, ,
\] 
where $(\cdot, \cdot)$ denotes the Euclidean scalar product.
\end{proof}

\section{Preliminaries on classical $T_N$ configurations}\label{TC}

This section is devoted to a slight generalization of a powerful machinery introduced in \cite{LSR} to study $T_N$ configurations. 

\subsection{Sz\'ekelyhidi's characterization of $T_N$ configurations in $\R^{2\times 2}$} We start with the following elegant characterization. 

\begin{prop}\label{PRO}(\cite[Proposition 1]{LSR})
Given a set $\{X_1,\dots,X_N\}\subset \R^{2\times 2}$ and  $\mu \in \R$, we let $A^\mu$ be the following $N\times N$ matrix:
\[
A^\mu :=
\begin{bmatrix}
    0 & \det(X_1 - X_2) & \det(X_1 - X_3) & \dots  & \det(X_1 - X_N) \\
   \mu\det(X_1 - X_2) & 0 & \det(X_2 - X_3) & \dots  & \det(X_2 - X_N) \\
    \vdots & \vdots & \vdots & \ddots & \vdots \\
    \mu\det(X_1 - X_N) & \mu\det(X_2 - X_N) & \mu\det(X_3 - X_N) & \dots  & 0
\end{bmatrix}.
\]
Then, $\{X_1,\dots, X_N\}$ induces a $T_N$ configuration if and only if there exists a vector $\lambda \in \mathbb{R}^N$ with positive components and  $\mu >1$ such that
\[
A^\mu\lambda = 0.
\]
\end{prop}

Even though not explicitely stated in \cite{LSR}, the following Corollary is part of the proof of Proposition \ref{PRO} and it is worth stating it here again, since we will make extensive use of it in the sequel. 

\begin{Cor}\label{coranalog}
Let $\{X_1,\dots,X_N\}\subset \R^{2\times 2}$ and let $\mu>1$ and $\lambda\in \mathbb R^N$ be a vector with positive entries such that $A^\mu \lambda =0$. Define the vectors
\begin{equation}\label{defnt}
t^{i} := \frac{1}{\xi_i}(\mu\lambda_1,\dots,\mu\lambda_{i - 1},\lambda_{i},\dots,\lambda_N), \text{ for } i \in \{1,\dots,N\}
\end{equation}
where $\xi_i>0$ is a normalizing constant so that $\|t^i\|_1 := \sum_j |t^i_j|= 1,\forall i$. Define the matrices $C_j$ with $j\in \{1, \ldots, N-1\}$ and $P$ by solving recursively 
\begin{equation}\label{sum}
\sum_{j = 1}^Nt_j^iX_j = P + C_1 + \dots + C_{i - 1}
\end{equation}
and set $C_N := - C_1- \ldots - C_{N-1}$. Finally, define
\begin{equation}\label{defnk}
k_i = \frac{\mu\lambda_1 + \dots + \mu\lambda_i + \lambda_{i + 1} \dots + \lambda_N}{(\mu - 1)\lambda_i}\, .
\end{equation}
Then $P, C_1, \ldots , C_N$ and $k_1, \ldots k_N$ give a $T_N$ configuration induced by $\{X_1, \ldots , X_N\}$
(i.e. \eqref{form} holds). 
 
Moreover, the following relation holds for every $i$:
\begin{equation}\label{sumdet}
\det\left(\sum_{j = 1}^Nt_j^iX_j\right) = \sum_{j=1}^Nt_j^i\det(X_j)\, .
\end{equation}
\end{Cor}

\begin{rem}\label{r:invert} Observe that the relations \eqref{defnk} can be inverted in order to compute $\mu$ and $\lambda$  (the latter up to scalar multiples) in terms of $k_1, \ldots , k_N$. In fact, let us impose 
\[
\|\lambda\|_1 = \lambda_1 + \cdots + \lambda_N =1\, .
\] 
Then, regarding $\mu$ as a parameter, the equations \eqref{defnk} give a linear system in triangular form which can be explicitely solved recursively, giving the formula 
\begin{align}
\lambda_j &= \frac{k_1 k_2 \cdots k_{j-1}}{(\mu-1)(k_1-1) (k_2-1) \cdots (k_j-1)}\, . \label{e:recursive}
\end{align}
The following identity can easily be proved by induction:
\[
\frac{1}{k_1-1} + \frac{k_1}{(k_1-1)(k_2-1)} + \cdots + \frac{k_1 \cdots k_{j-1}}{(k_1-1) \cdots (k_j-1)} = \frac{k_1\cdots k_j}{(k_1-1)\cdots (k_j-1)} -1\, .
\]
Hence, summing \eqref{e:recursive} and imposing $\sum_j \lambda_j =1$ we find the equation
\[
1 = \frac{1}{\mu-1} \left(\frac{k_1\cdots k_N}{(k_1-1) \cdots (k_N-1)}-1\right)\, ,
\]
which determines uniquely $\mu$ as
\begin{equation}\label{e:formula_mu}
\mu = \frac{k_1\cdots k_N}{(k_1-1) \cdots (k_N-1)}\, .
\end{equation}
\end{rem}

A second corollary of the computations in \cite{LSR} is that

\begin{Cor}\label{c:converse} Assume $\{X_1, \ldots , X_N\}\in \R^{2\times 2}$ induce the $T_N$ configuration of form \eqref{form} and let $\mu$ and $\lambda$ be as in \eqref{e:recursive} and \eqref{e:formula_mu}. Then $A^\mu \lambda =0$.
\end{Cor}

\subsection{A characterization of $T_N$ configurations in $\R^{n\times m}$} We start with a straightforward consequence of the results above.

Let us first introduce some notation concerning multi-indexes. We will use $I$ for multi-indexes referring to ordered sets of rows of matrices and $J$ for multi-indexes referring to ordered sets of columns. In our specific case, where we deal with matrices in $\R^{n\times m}$ we will thus have
\begin{align*}
I &= (i_1,\dots,i_r),\qquad 1\le i_1<\dots< i_r \le n\, ,\\
 \text{ and } \qquad J &= (j_1,\dots,j_s),\qquad 1\le j_1< \dots< j_s\le m\, 
\end{align*}
and we will use the notation $|I|:= r$ and $|J|:=s$. In the sequel we will always have $r = s$. 

\begin{Def}\label{multiind}
We denote by $\mathcal{A}_r$ the set
\[
\mathcal{A}_r = \{(I,J): |I| = |J| = r\},\qquad  2\le r \le \min(n,m) .
\]
For a matrix $M\in\R^{n\times m}$ and for $Z\in \mathcal{A}_r$ of the form $Z = (I,J)$, we denote by $M^Z$ the squared $r\times r$ matrix obtained by $A$ considering just the elements $a_{ij}$ with $i\in I$, $j\in J$ (using the order induced by $I$ and $J$). 

Given a set $\{X_1,\dots, X_N\}\subset \R^{n\times m}$, $\mu \in \R$ and $Z \in \mathcal{A}_r$, we introduce the matrix
\[
A_Z^\mu:=
\begin{bmatrix}
    0 & \det(X^Z_2 - X^Z_1) & \det(X^Z_3 - X^Z_1) & \dots  & \det(X^Z_N - X^Z_1) \\
   \mu\det(X^Z_1 - X^Z_2) & 0 & \det(X^Z_3 - X^Z_2) & \dots  & \det(X^Z_N - X^Z_2) \\
    \vdots & \vdots & \vdots & \ddots & \vdots \\
    \mu\det(X^Z_1 - X^Z_N) & \mu\det(X^Z_2 - X^Z_N) & \mu\det(X^Z_3 - X^Z_N) & \dots  & 0
\end{bmatrix}.
\]
\end{Def}

\begin{prop}\label{p:laszlo1}
A set $\{X_1, \ldots, X_N\}\subset \R^{n\times m}$ induces a $T_N$ configuration if and only if there is a real $\mu >1$ and a vector $\lambda\in \R^N$ with positive components such that
\[
A^\mu_Z \lambda = 0 \qquad \forall Z \in \mathcal{A}_2\, .
\]
Moreover, if we define the vectors $t^i$ as in \eqref{defnt}, the coefficients $k_i$ through \eqref{defnk} and the matrices $P$ and $C_i$ through \eqref{sum}, then $P, C_1, \ldots , C_N$ and $k_1, \ldots , k_N$ give a $T_N$ configuration induced by $\{X_1, \ldots , X_N\}$. 
\end{prop}

For this reason and in view of Remark \ref{r:invert}, we can introduce the following terminology:

\begin{Def}\label{d:defining_vector}
Given a $T_N$-configuration $P, C_1, \ldots , C_N$ and $k_1, \ldots, k_N$ we let $\mu$ and $\lambda$ be given by 
\eqref{e:recursive} and \eqref{e:formula_mu} and we call $(\lambda, \mu)\in \mathbb R^{N+1}$ the \emph{defining vector} of the $T_N$ configuration. 
\end{Def}

\begin{proof}[Proof of Proposition \ref{p:laszlo1}]
{\bf Direction $\Longleftarrow$.} Fix a set $\{X_1, \ldots, X_N\}$ of matrices with the property that there is a common $\mu>1$ and a common $\lambda$ with positive entries such that $A^\mu_Z \lambda =0$ for every $Z\in \mathcal{A}_2$. For each $Z$ we consider the corresponding set $\{X_1^Z, \ldots , Z_N^Z\}$ and we use the formulas \eqref{defnt}, \eqref{defnk} and \eqref{sum} to find $k_1, \ldots , k_N$, $P (Z)$ and $C_i (Z)$ such that
\[
X_i^Z = P (Z) + C_1 (Z) + \ldots C_{i-1} (Z) + k_i C_i (Z)\, .
\]
Since the coefficients $k_i$ are independent of $Z$, the formulas give that the matrices $C_i (Z)$ (and $P(Z)$) are compactible, in the sense that, if $j\ell$ is an entry common to $Z$ and $Z'$, then $(C_i (Z))_{j\ell} = (C_i (Z'))_{j\ell}$. In particular there are matrics $C_i$'s and $P$ such that $C_i (Z) = C_i^Z$ and $P (Z) = P^Z$ and thus \eqref{form} holds. Moreover, we also know from Proposition \ref{PRO} that $\rank (C_i^Z)\leq 1$ for every $Z$ and thus $\rank (C_i)\leq 1$. We also know that $C_1^Z+ \ldots + C_N^Z =0$ for every $Z$ and thus $C_1 + \ldots + C_N=0$. 

{\bf Direction $\Longrightarrow$}. Assume $X_1, \ldots , X_N$ induce a $T_N$ configuration as in \eqref{form}. Then $X_1^Z, \ldots , X_N^Z$ induce a $T_N$ configuration with corresponding $P^Z, C_1^Z, \ldots, C_N^Z$ and $k_1, \ldots , k_N$, where the latter coefficients are independent of $Z$. Thus, by Corollary \ref{c:converse}, $A^\mu_Z \lambda =0$. 
\end{proof}

\subsection{Computing minors} We end this section with a further generalization, this time of 
\eqref{sumdet}: we want to extend the validity of it to any minor.
  
\begin{prop}\label{analogm}
Let $\{X_1, \ldots, X_N\}\subset \R^{n\times m}$ induce a $T_N$ configuration as in \eqref{form} with defining vector 
$(\lambda, \mu)$. Define the vectors $t^1,\dots,t^N$ as in \eqref{defnt} and for every $Z\in \mathcal{A}_r$ of order $r \leq \min \{n,m\}$ define the 
minor $\mathcal{S} : \R^{n\times m} \ni X \mapsto \mathcal{S} (X) := \det (X^Z)\in \R$. Then 
\begin{equation}\label{e:sum_minor}
\sum_{j = 1}^Nt_j^i \mathcal{S}(X_j) = \mathcal{S}\left(\sum_{j = 1}^Nt_j^iX_j\right) = \mathcal{S}(P + C_1 + \dots + C_{i - 1})\, .
\end{equation}
and $A^\mu_Z \lambda = 0$.
\end{prop}

Fix any matrix $A\in \R^{m\times m}$. In the following we will denote by $\cof(A)$ the $m\times m$ matrix defined\footnote{Note that sometimes in the literature one refers to what we called $\cof(A)$ as the \emph{adjoint} of $A$, and the adjoint of $A$ would be $\cof(A)^T$. } as $$\cof(A)_{ij}:= (-1)^{i + j}\det(A^{j,i}),$$ where $A^{j,i}$ is the $m-1\times m - 1$ matrix obtained by eliminating from $A$ the $j$-th row and the $i$-th column. It is well-known that
\[
A\cof(A) = \cof(A)A = \det(A)\Id_{m}.
\] 
We will need the following elementary linear algebra fact, which in the literature is sometimes called Matrix Determinant Lemma:
\begin{lemma}\label{MDL}
Let $A,B$ be matrices in $\R^{m\times m}$, and let $\rank(B) \le 1$. Then,
\[
\det(A + B) = \det(A) + \langle\cof(A)^T,B\rangle
\]
\end{lemma}

Moreover, we need another elementary computation, which is essentially contained in \cite{LSR} and for which we report the proof at the end of the section for the reader's convenience.  

\begin{lemma}\label{l:linear}
Assume the real numbers $\mu>1$, $\lambda_1, \ldots , \lambda_N >0$ and $k_1, \ldots , k_N >1$ are linked by the formulas \eqref{defnk}. Assume $v, v_1, \ldots, v_N, w_1, \ldots , w_N$ are elements of a vector space satisfying the  relations
\begin{align}
w_i &= v + v_1 + \ldots + v_{i-1} + k_i v_i\\
0 &= v_1+ \ldots + v_N\, .
\end{align}
If we define the vectors $t^i$ as in \eqref{defnt}, then
\begin{equation}\label{e:sumfinale}
\sum_j t^i_j w_j = v + v_1 + \ldots + v_{i-1}\, .
\end{equation}
\end{lemma}

\begin{proof}[Proof of Proposition \ref{analogm}] 
Fix the $Z$ of the statement of the proposition. $X_1^Z, \ldots , X_N^Z$ induces $T_N$ with the same coefficients $k_1, \ldots k_N$.  This reduces therefore the statement to the case in which $m=n$, $Z = ((1, \ldots n), (1, \ldots , n))$ and the minor $\mathcal{S}$ is the usual determinant.  

We first prove \eqref{e:sum_minor}.
In order to do this we specialize \eqref{e:sumfinale} to $w_\ell = \det(X_\ell)$, $v = \det(P)$, $v_\ell =\langle \cof^T(P + C_1 + \dots + C_{\ell - 1}),C_\ell\rangle$. To simplify the notation set
\[
P^{(1)} = P, \text{ and } P^{(\ell)} =  P + C_1 + \dots + C_{\ell - 1}\qquad \forall\ell \in \{1,\dots, N + 1\}.
\]
We want to show that
\[
v + v_1 + \dots + v_{i - 1} = \det(P^{(i)}) \text{ and } v_1 + \dots + v_N = 0,
\]
and this would conclude the proof of \eqref{e:sum_minor} because of Lemma \ref{l:linear}. A repeated application of Lemma \ref{MDL} yields: 
\begin{align*}
&v + v_1 + \dots + v_{i - 1} = \underbrace{\underbrace{\det(P) + \langle \cof^T(P),C_1\rangle}_{\det(P^{(2)})} + \langle \cof^T(P^{(2)}),C_2\rangle}_{\det(P^{(3)})} +\\
& + \dots + \langle \cof^T(P^{(i)}),C_{i - 1}\rangle= \det(P^{(i)}) = \det(P + C_1 + \dots + C_{i - 1}).
\end{align*}
As a consequence of Lemma \ref{MDL}, we also have $v_\ell = \det(P^{(\ell + 1)}) - \det(P^{(\ell)})$. Therefore:
\begin{align*}
v_1 + \dots + v_N  = \sum_{\ell = 1}^N\left(\det(P^{(\ell + 1)}) - \det(P^{(\ell)})\right) = \det(P^{(N + 1)}) - \det(P^{(1)}).
\end{align*}
Since $\sum_\ell C_\ell =0$ and $\det(P^{(N +1)}) = \det(P + \sum_\ell C_\ell)$, the following holds
\[
\det(P^{(N + 1)}) - \det(P^{(1)}) = \det(P + \sum_\ell C_\ell) -\det(P) = \det(P) - \det(P) = 0,
\]
and the conclusion is thus reached.

To prove the second part of the statement notice that $A^\mu_Z\lambda = 0$ is equivalent to the following $N$ equations:
\[
\sum_{j = 1}^Nt_j^i\det(X_j - X_i) = 0\qquad\qquad\forall i \in \{1,\dots,N\}.
\]
Fix $i \in \{1,\dots,N\}$ and define matrices $Y_j := X_j - X_i, \forall j$. $\{Y_1,\dots,Y_N\}$ is still a $T_N$ configuration of the form 
\[
Y_i = P' + \sum_{\ell = 1}^{i - 1}C_\ell + k_iC_i,
\]
and $P' = -X_i$ (recall that $P = 0$). Apply now \eqref{e:sum_minor} to find that
\[
\sum_jt_j^i\det(X_j - X_i) = \sum_jt_j^i\det(Y_j) = \det\left(P' + \sum_{\ell = 1}^{i - 1}C_\ell\right) = \det\left(-X_i + \sum_{\ell = 1}^{i - 1}C_\ell\right) = \det(-k_iC_i) = 0
\]
and conclude the proof.
\end{proof}

\subsection{Proof of Lemma \ref{l:linear}} It is sufficient to compute separately $\sum_{j= 1}^Nt_j^1w_j=\sum_{j = 1}^N\lambda_jw_j$ and $\sum_{j=1}^{i - 1}\lambda_jw_j$. In fact,
\begin{equation}\label{goal}
\sum_j^Nt_j^iw_j =\frac{1}{\xi_i}\left[ \sum_{j = 1}^N\lambda_jw_j + (\mu - 1)\sum_{j=1}^{i - 1}\lambda_jw_j\right].
\end{equation}
We can write
\[
\sum_j\lambda_jw_j = v + a_1v_1 + \dots + a_Nv_N,
\]
being, $\forall \ell \in \{1,\dots,N\}, a_\ell = k_\ell\lambda_\ell + \dots + \lambda_N$. Recalling that the defining vector and the numbers $k_i$ are related through \eqref{defnk}, we compute
\begin{equation}\label{a}
\begin{split}
&a_\ell = k_\ell\lambda_\ell + \dots + \lambda_N = \frac{\mu \lambda_1 + \dots + \mu\lambda_{\ell} + \lambda_{\ell + 1} + \dots \lambda_N}{\mu - 1} + \lambda_{\ell + 1} +\dots + \lambda_N\\
& = \frac{\mu(\lambda_1 + \dots + \lambda_N)}{\mu - 1} = \frac{\mu}{\mu - 1} =: a.
\end{split}
\end{equation}
Hence
\[
\sum_{j = 1}^N\lambda_jw_j = v + \frac{\mu}{\mu - 1}(v_1 + \dots + v_N).
\]
On the other hand,
\[
\sum_{j = 1}^{i - 1}\lambda_jw_j = b_1v + b_2v_1 + \dots + b_iv_{i - 1},
\]
and
\begin{align*}
&b_1 = \lambda_1 + \dots + \lambda_{i - 1} =: c,\\
&b_\ell = k_\ell\lambda_\ell + \dots + \lambda_{i - 1} =\frac{\mu(\lambda_1 + \dots + \lambda_{\ell}) + \sum_{j = \ell + 1}^{N}\lambda_j + (\mu - 1) \sum_{j = \ell + 1}^{i - 1}\lambda_j }{\mu - 1}=\\
&=\frac{\mu( \sum_{j = 1}^{i - 1}\lambda_j) + \sum_{j = i}^{N}\lambda_j }{\mu - 1} =: b, \forall \ell\in\{2,\dots,i\}.
\end{align*}
Also, $$\xi_i = \|(\mu\lambda_1,\dots,\mu\lambda_{i - 1},\lambda_i,\dots,\lambda_N)\|_1= (\mu - 1)(\lambda_1 + \dots + \lambda_{i -1}) + 1 = (\mu - 1)b = 1 +(\mu - 1)c.$$ We can now compute \eqref{goal}:
\begin{align*}
&\frac{1}{\xi_i}\left[ \sum_{j = 1}^N\lambda_jw_j + (\mu - 1)\sum_{j=1}^{i - 1}\lambda_jw_j\right] =\\
&\frac{1}{\xi_i}\left[ v + a_1v_1 + \dots + a_Nv_N + (\mu - 1)(b_1v + b_2v_1 +\dots + b_iv_{i - 1})\right] =\\
&\frac{1}{\xi_i}\left[(\mu - 1)b(v + v_1 + \dots + v_{i - 1}) + a(v_1 + \dots + v_N)\right]=\\
&v + v_1 + \dots + v_{i - 1} + \frac{a}{(\mu - 1)b}(v_1 + \dots + v_N)
\end{align*}
We use the just obtained identity
\begin{equation}\label{sumfinale}
\sum_{j = 1}^{N}t_j^iw_j = v + v_1 + \dots + v_{i - 1} + \frac{a}{(\mu - 1)b}(v_1 + \dots + v_N)
\end{equation}
Using that $v_1+\ldots + v_N =0$ we conclude the desired identity.

\section{Inclusions sets relative to polyconvex functions}\label{PF}

In this section we consider the following question. Given a set of distinct matrices $A_i \in \R^{2n}\times \R^m$
\begin{equation}\label{inclusion}
A_i:=\left(
\begin{array}{c}
X_i\\
Y_i
\end{array}
\right)\, ,
\end{equation}
do they belong to a set of the form
\begin{equation}\label{e:inclusion2}
K'_f :=\left(
\begin{array}{c}
X_i\\
Df(X_i)\\
\end{array}
\right)
\end{equation}
for some strictly polyconvex function $f:\mathbb{R}^{n\times m}\to \mathbb{R}$?
Observe that $A_i \neq A_j,$ for $i \neq j$ if and only if $X_i\neq X_j$, for $i \neq j$. Below we will prove the following

\begin{prop}\label{p:convexity}
Let $f:\R^{n\times m}\to \R$ be a strictly polyconvex function of the form $f(X) =g(\Phi(X))$, where
$g \in C^1$ is strictly convex and $\Phi$ is the vector of all the subdeterminants of $X$, i.e.
\[
\Phi(X) = (X,v_1(X),\dots,v_{\min(n,m)}(X)),
\]
and $$v_s(X) = (\det(X_{Z_1}),\dots, \det(X_{Z_{\#\mathcal{A}_s}}))$$ for some fixed (but arbitrary) ordering of all the elements $Z\in\mathcal{A}_s$. If $A_i \in K'_f$ and $A_i \neq A_j$ for $i \neq j$, then $X_i$, $Y_i = D f(X_i)$ 
and $c_i = f (X_i)$ fulfill the following inequalities for every $i\neq j$:
\begin{multline}\label{finalemdim}
c_i - c_j +\langle Y_i,X_j - X_i\rangle
\\
 -\sum_{r = 2}^{\min(m,n)}\sum_{Z\in\mathcal{A}_r}d^i_{Z}\left(\langle{\cof}(X_i^Z)^T,X^Z_j - X^Z_i\rangle -\det(X_j^Z) + \det(X_i^Z)\right)<0, 
\end{multline}
where $d^i_Z = \partial_Zg(\Phi(X_i))$. 
\end{prop}

The expressions in \eqref{finalemdim} can be considerably simplied when the matrices $X_1, \ldots , X_N$ induce a $T_N$ configuration.

\begin{lemma}\label{l:minors}
Assume $X_1,\dots, X_N$ induces a $T_N$ configuration of the form $\eqref{form}$ and associated vectors $t^i$, $i\in \{1,\dots,N\}$. Then, $\forall i\in\{1,\dots,N\}$, $\forall r \in\{2,\dots,\min(m,n)\}$, $\forall Z\in\mathcal{A}_r$,
\begin{equation}\label{fzero}
\sum_jt_j^i\left(\langle{\cof}(X_i^Z)^T,X_j^Z - X_i^Z\rangle -\det(X_j^Z) + \det(X_i^Z)\right) = 0.
\end{equation}
\end{lemma}

In particular combining \eqref{finalemdim} and \eqref{fzero} we immediately get the following:

\begin{Cor}\label{absurd}
Let $f$ be a strictly polyconvex function and let $A_1, \ldots , A_N$ be distinct elements of $K'_f$ with the additional property that $\{X_1, \ldots , X_N\}_i$ induces a $T_N$ configuration of the form \eqref{form} with defining vector $(\mu, \lambda)$. Then,
\begin{equation}\label{e:absurd}
c_i - \sum_{j}t_j^ic_j -k_i\langle Y_i,C_i\rangle< 0 ,
\end{equation}
where the $t^i$'s are given by \eqref{defnt}. 
\end{Cor}

\subsection{Proof of Proposition \ref{p:convexity}}
The strict convexity of $g$ yields, for $i \neq j$,
\begin{equation}\label{mdim}
\langle Dg(\Phi(X_i)), \Phi(X_j) -\Phi(X_i)\rangle < g(\Phi(X_j)) - g(\Phi(X_i)).
\end{equation}
A simple computation shows that for the function $\det(\cdot):\R^{r\times r} \to \R$:
\[
D(\det(X))|_{X = Y} = \cof(Y)^T.
\]
In the following equation, we will write, for an $n\times m$ matrix $M$ and for $Z\in \mathcal{A}_r$, $\overline{{\cof}(M^Z)^T}$ to denote the $n\times m$ matrix with $0$ in every entry, except for the rows and columns corresponding to the multiindex $Z= (I, J)$, which will be filled with the entries of the matrix $\cof(M^Z)^T \in \R^{r\times r}$, namely, if $i\notin I$ or $j\not\in J$, then 
$(\overline{{\cof}(M^Z)^T})_{ij} = 0$ and, if we eliminate all such coefficients, the remaining $r\times r$ matrix equals $\cof (M^Z)^T$. Moreover, we will identify the differential of a map from $\R^{n\times m}$ to $\mathbb{R}$ with the obvious associated matrix.
We thus have the formula
\begin{align*}
Df(X) = D(g(\Phi(X))) = D_Xg(\Phi(X)) +\sum_{r = 2}^{\min(m,n)}\sum_{Z\in\mathcal{A}_r} \partial_Zg(\Phi(X))\overline{{\cof}(X^Z)^T}
\end{align*}
When evaluated on $X=X_i$, 
\[
Y_i = D_Xg(\Phi(X_i)) + \sum_{r = 2}^{\min(m,n)}\sum_{Z\in\mathcal{A}_r} \partial_Zg(\Phi(X_i))\overline{{\cof}(X_i^Z)}
\]
In order to simplify the notation set now $d^i_Z := \partial_Zg(\Phi(X_i))$. The previous expression yields:
\[
\begin{split}
\langle Dg(\Phi(X_i)),& \Phi(X_j) -\Phi(X_i)\rangle\nonumber\\ 
= & \langle D_Xg(\Phi(X_i)), X_j - X_i\rangle  + \sum_{r = 2}^{\min(m,n)}\sum_{Z\in\mathcal{A}_r}d^i_{Z}\left(\det(X_j^Z) - \det(X_i^Z)\right) \\
= &\left\langle Y_i - \sum_{r = 2}^{\min(m,n)}\sum_{Z\in\mathcal{A}_r} d^i_Z\overline{{\cof}(X_i^Z)^T},X_j - X_i\right\rangle + \sum_{r = 2}^{\min(m,n)}\sum_{Z\in\mathcal{A}_r}d^i_{Z}\left(\det(X_j^Z) - \det(X_i^Z)\right).
\end{split}
\]
Since
\[
g(\Phi(X_j)) - g(\Phi(X_i)) = f(X_j) - f(X_i) = c_j - c_i,
\]
\eqref{mdim} becomes:
\begin{align*}
\langle Y_i,X_j - X_i\rangle - \sum_{r = 2}^{\min(m,n)}\sum_{Z\in\mathcal{A}_r}d^i_{Z}\left(\langle\overline{{\cof}(X_i^Z)^T},X_j - X_i\rangle -\det(X_j^Z) + \det(X_i^Z)\right) < c_j - c_i.
\end{align*}
Finally, summing $c_i - c_j$ on both sides:
\begin{equation}
\begin{split}
& c_i - c_j +\langle Y_i,X_j - X_i\rangle  - \sum_{r = 2}^{\min(m,n)}\sum_{Z\in\mathcal{A}_r}d^i_{Z}\left(\langle\overline{{\cof}(X_i^Z)^T},X_j - X_i\rangle -\det(X_j^Z) - \det(X_i^Z)\right) < 0
\end{split}
\end{equation}
Using the fact that $\langle\overline{{\cof}(X_i^Z)^T},X_j - X_i\rangle = \langle {\cof}(X_i^Z)^T,X_j^Z - X_i^Z\rangle$, we see that the previous inequality implies the conclusion
\begin{equation*}
\begin{split}
&\forall i \neq j,\\
&c_i - c_j +\langle Y_i,X_j - X_i\rangle  -\sum_{r = 2}^{\min(m,n)}\sum_{Z\in\mathcal{A}_r}d^i_{Z}\left(\langle{\cof}(X_i^Z)^T,X_j^Z - X_i^Z\rangle -\det(X_j^Z) + \det(X_i^Z)\right) <0.
\end{split}
\end{equation*}

\subsection{Proof of Lemma \ref{l:minors}}
The result is a direct consequence of  Lemma \ref{MDL} and Proposition \ref{analogm}. 
First of all, by Proposition \ref{analogm} we have
\begin{equation}\label{e:fzerohilfe1}
\sum_jt_j^i\det(X_j^Z) = \det\left(\sum_jt_j^iX_j^Z\right) = \det\left(P_1^Z + \dots + C_{i - 1}^Z\right)
\end{equation}
Moreover, by \eqref{sum}, we get
\begin{equation}\label{e:fzerohilfe2}
 \sum_jt_j^i\langle {\cof}(X_i^Z)^T,X_j^Z - X_i^Z\rangle = \langle {\cof}(X_i^Z)^T,P^Z +C^Z_1 +\dots + C^Z_{i - 1} - X_i^Z\rangle = -k_i\langle{\cof}(X_i^Z)^T,C_i^Z\rangle\, .
 \end{equation}
Finally, apply Lemma \ref{MDL} to $A= X_i^Z$ and $B= - k_i C_i^Z$ to get
\begin{equation}\label{e:fzerohilfe3}
\det\left(P^Z + \dots + C_{i - 1}^Z\right) = \det (X_i^Z) - k_i\langle\cof(X_i^Z)^T, C_i^Z \rangle\, .
\end{equation}
These three equalities together give \eqref{fzero}. 

\subsection{Proof of Corollary \ref{absurd}} Multiply \eqref{finalemdim} by $t^i_j$ and sum over $j$. Using Lemma \ref{l:minors} and taking into account $\sum_j t^i_j=1$ we get
\[
c_i - \sum_j t^i_j  c_j + \left\langle Y_i, \sum_j t^i_j X_j - X_i \right\rangle < 0\, . 
\]
Since
\[
\sum_j  t^i_j X_j = P + C_1 + \ldots + C_{i-1}
\]
and
\[
X_i = P + C_1 + \ldots + C_{i-1} + k_i C_i\, ,
\]
we easily conclude \eqref{e:absurd}.

\section{Proof of Theorem \ref{t:main}}\label{MT}

In this section we prove the main theorem of this paper. 

\subsection{Gauge invariance} In the first part we state a corollary of some obvious invariance of polyconvex functions under certain groups of transformations. This invariance will then be used in the proof of Theorem \ref{t:main} to bring an hypothetical $T'_N$ configuration into a ``canonical form''.

\begin{lemma}\label{simplification}
Let $f: \R^{n\times m}$ be strictly polyconvex and assume that $K_f$ contains a set of matrices $\{A_1, \ldots , A_N\}$ which induces a nondegenerate $T'_N$ configuration, denoted by 
\[
A_i :=
\left(
\begin{array}{c}
\bar{X}_i\\
\bar{Y}_i\\
\bar{Z}_i
\end{array}
\right),
\]
where
\begin{align*}
\bar{X}_i = P + \bar{C}_1 + \dots + k_i\bar{C}_i,\\
\bar{Y}_i = Q + \bar{D}_1 + \dots + k_i\bar{D}_i,\\
\bar{Z}_i = R + \bar{E}_1 + \dots + k_i\bar{E}_i.
\end{align*}
Then, for every $S,T \in \R^{n\times m}, a\in \R$, there exists another strictly polyconvex function $\bar f$ such that the family of matrices 
\[
B_i :=
\left(
\begin{array}{c}
X_i\\
Y_i\\
Z_i
\end{array}
\right)
\]
lie in $K_{\bar f}, \forall i$, and they have the following properties: 
\begin{itemize}
\item The matrices $X_i$, $Y_i$ have the form
\begin{align*}
&X_i = S +\bar{C}_1 + \dots + k_i\bar{C}_i,\\
&Y_i = T + \bar{D}_1 + \dots + k_i\bar{D}_i.
\end{align*}
\item the matrices $Z_i$ are of the form
\[
Z_i = U + E_1 + \dots + k_iE_i,
\]
where $$U = R - P^T( Q - T) + (S - P)^TN T+ (\langle P, Q - T\rangle + a)\id.$$ Moreover, if $n_i \in \R^m$ is the unit vector of Proposition \ref{p:T_N'_easy}, then we have
\begin{align}
&\sum_i E_i = 0, \label{firstreq}\\
&\sum_jt_j^i Z_i = U + E_1 + \dots + E_{i - 1}, \forall i\in\{1,\dots,  N\}\label{secreq},\\
&E_in_i = 0,\forall i \label{thirreq}.
\end{align}
\end{itemize}
\end{lemma}

\begin{proof}
We consider $\bar f$ of the form
\[
f(X) = \bar f(X + O) + \langle X, V\rangle + a.
\]
We want $X_i - \bar{X}_i = S - P$ and $Y_i - \bar{Y}_i = T - Q$, therefore the natural choice for $O$ is $O := -P + S$. In this way,
\[
f(\bar{X_i}) = \bar f(\bar{X}_i - P + S) + \langle\bar{X}_i,V\rangle + a = \bar f(X_i) + \langle\bar{X}_i,V\rangle + a.
\]
Moreover, we have
\[
Df(\bar{X}_i) = \bar{Y}_i,
\]
hence $D\bar f({X}_i) = Y_i $ if and only if
\[
Df(\bar{X}_i) - V = \bar{Y}_i - V = Y_i,
\]
i.e. $V := Q - T$. We now show that the modification of $\bar{Z}_i$ into $Z_i$ with the properties listed in the statement of the present proposition will let us fulfill also the last requirement, namely that
\[
Z_i = X_i^TY_i - c'_i\id,
\]
where $c'_i = \bar f(X_i)$. Analogously, we denote with $c_i = f(\bar{X}_i)$. We write
\begin{align*}
&X_i^TY_i - c'_i\id = (X_i - O)^TY_i + O^TY_i - (c_i - \langle \bar{X}_i,V\rangle - a)\id =\\
&\bar{X}_i^T(Y_i + V) - \bar{X}_i^TV + O^TY_i - (c_i - \langle \bar{X}_i,V\rangle - a)\id =\\
&\underbrace{\bar{X}_i^T\bar{Y}_i - c_i\id}_{\bar{Z_i}} - \bar{X}_i^TV + O^TY_i + (\langle \bar{X}_i,V\rangle + a)\id.
\end{align*}

We can thus rewrite
\[
\bar{X}_i^TV = P^TV + \sum_{j = 1}^{i - 1}C_j^TV + k_iC_i^TV.
\]
For every fixed $j$, we decompose in a unique way $V = V_j + V^\perp_j$, where $V_j^\perp  = (Vn_j)\otimes n_j$ and \(V_j=V-V_j^\perp\). Note that, since \(C_j=u_j\otimes v_j\),  this implies that $C_j^TV_j-\langle C_j,V_j\rangle\id $ is a scalar multiple of the orthogonal projection on $\spann(n_j)^\perp$. Therefore,
\[
C_j^TV = C_j^TV_j + C_j^TV_j^\perp = \underbrace{C_j^TV_j - \langle C_j,V_j\rangle\id + C_j^TV_j^\perp}_{=: R_j} + \langle C_j,V_j\rangle\id.
\]
Consequently, $\bar{X}_i^TV$ has the following form:
\[
\bar{X}_i^TV = P^TV + \sum_{j = 1}^{i - 1}R_j + k_iR_i + \left(\sum_{j=1}^{i - 1}\langle C_j,V_j\rangle + k_i\langle C_i,V_i\rangle\right)\id.
\]
Finally, we define, $Z_i' := P^TV + \sum_{j = 1}^{i - 1}R_j + k_iR_i$. Resuming the main computations, we have obtained that:
\[
X_i^TY_i - c'_i\id  = \bar{Z_i} - Z_i' + O^TY_i + (-\sum_{j = 1}^{i - 1}\langle C_j,V_j\rangle - k_i\langle C_i,V_i\rangle + \langle \bar{X}_i,V\rangle + a)\id.
\]
Since  $$-\sum_{j = 1}^{i - 1}\langle C_j,V_j\rangle - k_i\langle C_i,V_i\rangle + \langle \bar{X}_i,V\rangle =-\sum_{j = 1}^{i - 1}\langle C_j,V\rangle - k_i\langle C_i,V\rangle + \langle \bar{X}_i,V\rangle= \langle P,V \rangle,$$ we are finally able to say that the first part of the Proposition is proved provided that
\begin{align*}
&Z_i := \bar{Z_i} - Z_i' + O^TY_i + (\langle P,V\rangle + a)\id,\\
&E_i := \bar{E_i} - C_i^TV_i + \langle C_i,V_i\rangle\id - C_i^TV_i^\perp + O^TD_i,\\
&U:= R - P^TV + O^TT + (\langle P,V\rangle + a)\id.
\end{align*}
To simplify future computations, let us use the identities $V_i + V_i^\perp = V$ and $ \langle C_i,V_i\rangle\ =  \langle C_i,V\rangle$:
\begin{align*}
&Z_i := \bar{Z_i} - Z_i' + O^TY_i + (\langle P,V\rangle + a)\id,\\
&E_i := \bar{E_i} - C_i^TV + \langle C_i,V\rangle\id + O^TD_i,\\
&U:= R - P^TV +  O^TT + (\langle P,V\rangle + a)\id.
\end{align*}
Properties \eqref{firstreq}-\eqref{secreq}-\eqref{thirreq} are easily checked by the linearity of the previous expressions and the identity \eqref{sumfinale}.
\end{proof}

\subsection{Proof of Theorem \ref{t:main}}
Assume by contradiction the existence of a $T_N'$ configuration induced by matrices $\{A_1, \ldots , A_N\}$ which belong to the inclusion set $K_f$ of some stictly polyconvex function $f\in C^1 (\R^{n\times m})$. Note that the corresponding $\{X_i\}$ must be all distinct, because $Y_i = Df (X_i)$ and $Z_i = X_i^T Df (X_i) - f(X_i) {\rm id}$. Thus $\{X_1, \ldots , X_N\}$ induce a $T_N$ configuration.

We consider coefficients $k_1, \ldots , k_N$ and matrices $P, Q, R$, $C_i, D_i, E_i$ as in Definition \ref{Def_TN'}. By Lemma \ref{simplification} we can assume, without loss of generality, that 
\[
P = 0 = Q \text{ and } \tr(R) = 0\, .
\]
We are now going to prove that the system of inequalities
\begin{equation}\label{maininequalities}
-\nu_i := c_i - \sum_{j}t_j^ic_j -k_i\langle Y_i,C_i\rangle< 0, \forall i \, ,
\end{equation}
where $c_i$ and $t_j^i$ are as in Corollary \ref{absurd}, cannot be fulfilled at the same time. This will then give a contradiction. In order to follow our strategy, we need to compute the following sums:

\begin{equation}\label{finalsum?}
\sum_jt_j^iZ_j = \sum_jt_j^iX^T_jY_j - \sum_jt_j^ic_j\id.
\end{equation}

Let us start computing the sum for $i = 1$, $\sum_j\lambda_jX_j^TY_j.$ We rewrite it in the following way:

\begin{equation}\label{quadsum}
\begin{split}
 \sum_j\lambda_j X_j^TY_j =& \sum_{j = 1}^N\lambda_j\left(\sum_{1\le a,b\le j - 1}C_a^TD_b + k_j\sum_{1\le a \le j - 1} C_a^TD_j + k_j\sum_{1\le b \le j - 1} C_j^TD_b + k_j^2C_j^TD_j\right) \\
= & \sum_{i,j}g_{ij}C_i^TD_j,
\end{split}
\end{equation}
where we collected in the coefficients $g_{ij}$ the following quantities:
\[
g_{ij}=
\begin{cases}
 \lambda_ik_i + \sum_{r = i + 1}^N\lambda_r,\text{ if } i \neq j\\
\lambda_ik_i^2 + \sum_{r = i + 1}^N\lambda_r,\text{ if } i = j.
\end{cases}
\]
As already computed in \eqref{a}, we have:
\[
g_{ij} = g_{ji} = \lambda_ik_i + \sum_{r = i + 1}^N\lambda_r = \frac{\mu}{\mu - 1},
\]
On the other hand,
\[
g_{ii} = k_i^2\lambda_i + \sum_{r = i + 1}^N\lambda_r = k_i(k_i - 1)\lambda_i + \frac{\mu}{\mu - 1}.
\]
Using the equalities $\sum_\ell C_\ell = 0 = \sum_\ell D_\ell$, then also $\sum_{i,j}C_i^TD_j = 0$, and so $\sum_{i\neq j}C_i^TD_j = -\sum_iC_i^TD_i$. Hence, \eqref{quadsum} becomes
\[
 \sum_{i,j}g_{ij}C_i^TD_j = \frac{\mu}{\mu - 1}\sum_{i\neq j}C_i^TD_j + \sum_i\left( k_i(k_i - 1)\lambda_i + \frac{\mu}{\mu - 1}\right)C_i^TD_i =  \sum_i k_i(k_i - 1)\lambda_iC_i^TD_i.
\]
We just proved that

\begin{equation}\label{b}
\sum_j\lambda_jX_j^TY_j =  \sum_i k_i(k_i - 1)\lambda_iC_i^TD_i.
\end{equation}

In particular,
\begin{equation}\label{c}
\sum_j\lambda_j\langle X_j,Y_j\rangle = 0,
\end{equation}
since $C_i^TD_i$ is trace-free for every $i$. We also have:
\[
\sum_j\lambda_jZ_j = \sum_j\lambda_jX^T_jY_j - \sum_j\lambda_jc_j\id \Rightarrow \sum_i k_i(k_i - 1)\lambda_iC_i^TD_i = R + \sum_j\lambda_jc_j\id.
\]
Since both $\tr(R)$ and  $\tr\left(\sum_i k_i(k_i - 1)\lambda_iC_i^TD_i\right) =0$, then $ \sum_j\lambda_jc_j = 0$ and we get
\[
\sum_i k_i(k_i - 1)\lambda_iC_i^TD_i = R.
\]
Recall the definition of $t^i$, namely 
$$t^i= \frac{1}{\xi_i}(\mu\lambda_1,\dots,\mu\lambda_{i - 1},\lambda_i,\dots,\lambda_N)\, .$$
By the previous computation ($i = 1$) and \eqref{secreq}, it is convenient to rewrite \eqref{finalsum?} as
\begin{equation}\label{rew}
R + \sum_j^{i - 1}E_j = \frac{1}{\xi_i}\left(R + (\mu - 1)\sum_{j = 1}^{i - 1}\lambda_jX_j^TY_j\right) - \sum_jt_j^ic_j\id.
\end{equation}
Once again, let us express the sum up to $i - 1$ in the following way:
\[
\sum_{j = 1}^{i - 1}\lambda_jX_j^TY_j = \sum_{k,j}^{i - 1}s_{kj}C_k^TD_j.
\]
A combinatorial argument analogous to the one in the previous case gives
\begin{align*}
&
s_{\ell\ell}= k_\ell^2\lambda_\ell + \dots + \lambda_{i - 1}
 \\
&\phantom{s_{\ell\ell}}= (k_\ell^2 - k_\ell)\lambda_\ell + k_\ell\lambda_\ell  + \dots + \lambda_{i - 1},
\\
& s_{\alpha\beta} = k_\alpha\lambda_\alpha + \dots + \lambda_{i - 1},\qquad\alpha>\beta\\
& s_{\beta\alpha} = k_\beta\lambda_\beta + \dots + \lambda_{i - 1},\qquad \alpha<\beta.
\end{align*}
Now
\[
 k_r\lambda_r + \dots + \lambda_{i - 1}=\frac{\mu( \sum_{j = 1}^{i - 1}\lambda_j) + \sum_{j = i}^{N}\lambda_j }{\mu - 1}
\]
and so
\[
 k_r\lambda_r + \dots + \lambda_{i - 1}=\frac{(\mu -1)( \sum_{j = 1}^{i - 1}\lambda_j) + 1 }{\mu - 1} = \frac{\xi_i}{\mu - 1} =: b_{i - 1}
\]
Hence
\[
\sum_{j = 1}^{i - 1}\lambda_jX_j^TY_j = \sum_{k,j}^{i - 1}s_{kj}C_k^TD_j = b_{i - 1}\sum_{k,j}^{i - 1}C_k^TD_j + \sum_{\alpha = 1}^{i - 1} k_\alpha(k_\alpha - 1)\lambda_\alpha C_\alpha^TD_\alpha
\]

We rewrite \eqref{rew} as

\begin{equation}\label{useful}
R + \sum_{j = 1}^{i - 1}E_j = \frac{1}{\xi_i}\left(R + \xi_i\sum_{k,j}^{i - 1}C_k^TD_j + (\mu - 1)\sum_{\alpha = 1}^{i - 1} k_\alpha(k_\alpha - 1)\lambda_\alpha C_\alpha^TD_\alpha\right) - \sum_jt_j^ic_j\id
\end{equation}

$E_i$ is readily computed using \eqref{useful} and the definition of $Z_i$:
\[
k_iE_i + \frac{1}{\xi_i}\left(R + \xi_i\sum_{k,j}^{i - 1}C_k^TD_j + (\mu - 1)\sum_{\alpha = 1}^{i - 1} k_\alpha(k_\alpha - 1)\lambda_\alpha C_\alpha^TD_\alpha\right) - \sum_jt_j^ic_j\id = X_i^TY_i - c_i\id
\]
then
\begin{align*}
&k_iE_i + \frac{1}{\xi_i}\left(R + (\mu - 1)\sum_{\alpha = 1}^{i - 1} k_\alpha(k_\alpha - 1)\lambda_\alpha C_\alpha^TD_\alpha\right) =\\
&k_i\sum_j^{i - 1}C_i^TD_j +  k_i\sum_{j}^{i - 1}C_j^TD_i + k_i^2C_i^TD_i - c_i\id + \sum_jt_j^ic_j\id.
\end{align*}

The evaluation of the previous expression at the vectors $n_i$ of Proposition \ref{p:T_N'_easy}(ii) yields
\begin{equation}\label{finaleq}
\frac{1}{\xi_i}\left(Rn_i + (\mu - 1)\sum_{\alpha = 1}^{i - 1} k_\alpha(k_\alpha - 1)\lambda_\alpha C_\alpha^TD_\alpha n_i\right) =
k_i\sum_j^{i - 1}C_i^TD_jn_i  - c_in_i + \sum_jt_j^ic_jn_i.
\end{equation}
Now, since $C_iv = 0, \forall v\perp n_i$, we must have
\[
C_i^TD_jn_i = \langle C_i,D_j\rangle n_i.
\]
The last equality implies that the right hand side of \eqref{finaleq} is exactly $\nu_i n_i$, where $\nu_i$ has been defined in \eqref{maininequalities}. We will now prove that there exists a nontrivial subset $A\subset \{1,\dots, N\}$ such that
\[
\sum_{j \in A}\xi_j\nu_j =0,
\]
and this will conclude the proof, being $\xi_j > 0,\forall j$. Since $C_\alpha^TD_\alpha n_i = b_{\alpha i}n_\alpha$, if we define 
\[
a_{\alpha i} := -(\mu - 1)\xi_ik_\alpha(k_\alpha - 1)\lambda_\alpha b_{\alpha i}\, ,
\] 
then we can rewrite \eqref{finaleq} as
\[
Rn_i =\xi_i\nu_i n_i + \sum_{\alpha = 1}^{i - 1}a_{\alpha i}n_\alpha.
\]
Now, consider the set $A \subset \{1,\dots,N\}$ defined as
\[
A = \{1\}\cup\{j: n_j \text{ cannot be written as a linear combination of vectors } n_\ell, \text{ for any }\ell\le j\}.
\]
Clearly 
\[
\spann (\{n_s: s\in A\}) = \spann(n_1,\dots,n_N) \subset \mathbb{R}^m
\] 
and moreover $\{n_s: s\in A\}$ are linearly independent.
Define $S:= \spann(\{n_s:s\in A\})$, and consider the relation
\[
Rn_i =\xi_i\nu_i n_i + \sum_{\alpha = 1}^{i - 1}a_{\alpha i}n_\alpha.
\]
for $i \in A$. This can be rewritten as
\begin{equation}\label{triang}
Rn_i =\xi_i\nu_i n_i + \sum_{\alpha \in A,\alpha\le{i - 1}}d_{\alpha i}n_\alpha,
\end{equation}
for some coefficients $d_{\alpha i}$.
 Recall that
\[
R = \sum_i k_i(k_i - 1)\lambda_iC_i^TD_i.
\]
By the properties of the matrices $C_i$'s, we see that $\im(R) \subseteq S$. Now complete (if necessary) $\{n_s: s \in A\}$ to a basis $\mathcal{B}$ of $\mathbb{R}^m$ adding vectors $\gamma_j$ with the property that $(\gamma_j,\gamma_k) = (\gamma_j,n_s) = 0,\forall j\neq k, s \in A$ and $\|\gamma_j\| = 1, \forall j$. By the previous observation about the image of $R$ and \eqref{triang}, we are able to write the matrix of the linear map associated to $R$ for the basis $\mathcal{B}$ as
\[
\left[
\begin{array}{c|c}
\begin{matrix}
    \xi_{i_1}\nu_{i_1} &* &* & \dots  &* \\
    0 & \xi_{i_2}\nu_{i_2} &* & \dots  &* \\
    \vdots & \vdots & \vdots & \ddots & \vdots \\
    0 & 0 &0 & \dots  &  \xi_{i_{\dim(S)}}\nu_{i_{\dim(S)}}
\end{matrix}

& \mathbf{T}\\ \hline
\mathbf{0}_{m - \dim(S),\dim(S)} & \mathbf{0}_{m - \dim(S),m - \dim(S)}.
\end{array}
\right]
\] 
We denoted with $\mathbf{0}_{a,b}$ the zero matrix with $a$ rows and $b$ columns, with $\mathbf{T}$ the $\dim(S)\times (m - \dim(S))$ matrix of the coefficients of $R\gamma_j$ with respect to $\{n_s:s\in A\}$, and with $\mathbf{*}$ numbers we are not interested in computing explicitely. Moreover, we chose an enumeration of $A$ with $1 = i_1 <i_2< \dots < i_{\dim(S)}$. The previous matrix must have the same trace as $R$, so
\[
0 = \tr(R) = \sum_{j = 1}^{\dim (S)}\xi_{i_j}\nu_{i_j},
\] 
and the proof is finished.

\section{Stationary graphs and stationary varifolds}\label{s:gmt}

The aim of this section is to provide the link between stationary points for energies defined on functions (or graphs) and stationary varifolds for "geometric" energies. 

\subsection{Notation and preliminary definitions}

Recall that general $m$-dimensional varifolds in $\R^{m+n}$ (introduced by L.C. Young in \cite{LCY} and pioneered in geometric measure theory by Almgren \cite{Alm} and Allard \cite{ALLARD}) are nonnegative Radon measures on the Grassmaniann of $\mathbb G (m, m + n)$ of (unoriented) $m$-dimensional planes of $\R^{m+n}$. In our specific case we are interested on a subclass, namely integer rectifiable varifolds, for which we can take the simpler Definition \ref{d:varifolds} below.
A quick reference for the terminology used in this section is \cite{CAM}, whereas comprehensive introductions can be found in the foundational paper \cite{ALLARD} and in the book \cite{SIM}. 

\begin{Def}\label{d:varifolds}
An \emph{integer rectifiable varifold} $V$ of dimension $m$ is a couple $(\Gamma,\theta)$, where $\Gamma\subset \R^{m + n}$ is a $m$-rectifiable set in $\R^N$, and $\theta:\Gamma \to \mathbb{N}\setminus\{0\}$ is a Borel map.
\end{Def}

It is customary to denote $(\Gamma , \theta)$ as $\theta \llbracket \Gamma \rrbracket$ and to call $\theta$ the multiplicity of the varifold. 

\begin{Def}
Let $U$ be an open set of $\R^{m + n}$, and let $\Phi: \R^{m + n} \to U$ be a diffeomorphism. The \emph{pushforward} of an integer rectifiable varifold $V = \theta \llbracket \Gamma \rrbracket$ through $\Phi$ is defined as $\Phi_\#V = \theta\circ \Phi^{-1} \llbracket \Phi(\Gamma) \rrbracket$.
\end{Def}

For an integer rectifiable varifold $\theta\llbracket\Gamma\rrbracket$, it is customary to introduce a notion of approximate tangent plane, which exists for $\mathcal{H}^m$-a.e. point of $\Gamma$,  we refer to \cite[Theorem 3.1.8]{SIM} for the relevant details. Provided it exists, the tangent plane at the point $y \in \Gamma$ will be denoted with $T_y\Gamma$ and it is an element of  
$\mathbb{G}(m,m+n)$. In the following, we will identify the Grassmanian manifold with a suitable subset of orthogonal projections, i.e. for every $L\in \mathbb{G} (m, m+n)$ we consider the linear map $P: \R^{m+n} \to \R^{m+n}$ which is the orthogonal projection onto $L$. With this identification we have
\[
\mathbb{G}(m,m+n) \sim \left\{P \in \R^{(m + n) \times (m + n)}: P =P^T, P^2 = P, \rank(P) = \tr(P) = m \right\}.
\]
Since we are interested in graphs, we introduce the following useful notation

\begin{Def} The set $\mathbb{G}_0 (m, m+n)$ is given by those $m$-dimensional planes $L$ which are the graphs of a linear map $X:\R^m \to \R^n$. Namely, if we regard $X$ as an element of $\R^{n\times m}$, $L = \{(x, X x): x\in \mathbb R^m\}\in \mathbb{G} (m, m+n)$.
\end{Def}

Regarding $X$ as an element of $\R^{n\times m}$, the orthogonal projection onto $L$, regarded as an element $h (X)$ of $\R^{(m+n)\times (m+n)}$, is then given by the formula 
\[
h(X):= M(X)S(X)M(X)^T
\]
where
\[
 M(X):=
\left(
\begin{array}{c}
\id_m\\
X
\end{array}
\right) \qquad \mbox{and}\qquad S(X):=(M(X)^TM(X))^{-1},
\]
or, more explicitely,
\begin{equation}\label{hform}
h(X) = \left[\begin{array}{c|c}
h_1(X) &h_3(X)\\ \hline
h_2(X)& h_4(X)
\end{array}\right] = \left[\begin{array}{c|c}
S(X)&S(X)X^T\\ \hline
XS(X)&XS(X)X^T
\end{array}\right].
\end{equation}
The map $h$ is a smooth diffeomorphism between $\R^{n\times m}$ and the open subset $\mathbb{G}_0$. We will use in general, i.e. for any matrix $M \in \R^{(m+n)\times (m+n)}$ the same splitting as in \eqref{hform}:
\begin{equation}\label{splitting}
M =
\left[
\begin{array}{c|c}
M_1 & M_3\\ \hline
M_2 & M_4
\end{array}
\right]
\end{equation}
with $M_1 \in \R^{m\times  m}$, $M_4 \in \R^{n\times n}$. In this section, we will use freely the following fact. Recall that, by \eqref{e:area}, for every $X\in \R^{n\times m}$ the area element is given by 
\[
\mathcal{A}(X) = \sqrt{\det(\id_m + X^TX)}.
\]
By the Cauchy-Binet formula, \cite[Proposition 2.69]{AFP},
\[
\mathcal{A}(X) = \sqrt{1 + \|X\|^2 + \sum_{r = 2}^{\min\{m,n\}}\sum_{Z \in \mathcal{A}_r}\det(X^Z)^2},
\]
where we  used the  notation  introduced in Definition \ref{multiind}.

Finally, throughout the section, we use the following notation:
\begin{itemize}
\item if $z \in \R^m\times\R^n$, then we will write $z = (x,y), x \in \R^m, y \in \R^n$;
\item $\pi :\R^m\times\R^n \to \R^m$ denotes the projection on the first factor, i.e. $\pi(z) = \pi((x,y)) = x$.
\end{itemize}

\subsection{Graphs and varifolds} If $u\in W^{1,p} (\Omega, \R^n)$, $\Omega\subset \R^m$ and $p>m$, Morrey's embedding theorem shows the existence of a precise representative of $u$ which is H\"older continuous. In what follows we will always assume that the map $u$ is given pointwise by such (H\"older) continuous precise representative. Correspondingly we introduce the notation $\Gamma_u$ for the (set-theoretic) graph $\{(x, u(x)):x\in \Omega\}$, which is a relatively closed subset of $\Omega\times \R^n$. The classical area formula (see for instance  \cite[Cor. 2, Ch. 3]{MGS1}) implies that $\Gamma_u$ is $m$-rectifiable and its $\mathcal{H}^m$ measure is given by
\[
\int_\Omega \mathcal{A} (Du)\, .
\]
We can thus consider the corresponding varifold $\llbracket \Gamma_u \rrbracket$.
  
If $u \in W^{1,m}(\Omega,\R^n)$, then $u$ has a precise representative which is however defined only up to a set of $m$-capacity $0$ (but not {\em everywhere}). Moreover, if for maps $u\in W^{1,m} \cap C (\Omega ,\R^n)$, for which the set-theoretic graph $\Gamma_u$ could be defined classically, it can be proven that $\Gamma_u$ does not necessarily have locally finite $\mathcal{H}^m$-measure, in spite of the fact that $\mathcal{A} (Du)$ belongs to $L^1_{loc}$. In particular the area formula fails. For this reason, following the notation and terminology of \cite[Sec. 1.5, 2.1]{MGS1}, we introduce the ``rectifiable part of the graph of $u$'', which will be denoted by $\mathcal{G}_u$ (the notation in \cite{MGS1} is in fact $\mathcal{G}_{u, \Omega}$: we will omit the domain $\Omega$ since in our case it is always clear from the context). 

First we denote the set of Lebesgue points of $u$ by $\mathcal{L}_u$ and we introduce the set
\[
A_D(u):=\{x \in \Omega: u \text{ is approximately differentiable at }x\}.
\] 
For the definition of approximate differentiability, see \cite[Sec. 1.4, Def. 3]{MGS1}. We also set
\[
\mathcal{R}_u := A_D(u)\cap \mathcal{L}_u.
\]
Notice that, since $u \in W^{1,m}(\Omega,\R^n)$, then $|\Omega\setminus \mathcal{R}_u| = 0$. From now on, we always assume that $u$ so that $u(x)$ is the Lebesgue value at every point $x\in \mathcal{R}_u$. The \emph{rectifiable part of the graph} of $u$ is then
\[
\mathcal{G}_u := \{(x,u(x)): x \in \mathcal{R}_u\}\, .
\]
By \cite[Sec. 1.5, Th. 4]{MGS1}, $\mathcal{G}_u$ is $m$-rectifiable and 
\[
\mathcal{H}^m(\mathcal{G}_{u}) = \int_{\Omega}\mathcal{A}(Du(x))\dx\, .
\]
Since $\mathcal{A} (Du)\in L^1_{loc}$, this allows us to introduce the integer rectifiable varifold \footnote{In fact the $\mathcal{G}_{u}$ can be oriented to give an integer rectifiable current of multiplicity $1$ and without boundary in $\Omega \times \R^n$, see \cite[Pr. 1, Sec 2.1]{MGS1}. The varifold that we consider is then the one induced by the current in the usual sense.} $\llbracket \mathcal{G}_u \rrbracket$. When $u\in W^{1,p}$ for $p>m$, the Lusin property (namely the fact that $v (x):= (x, u(x))$ maps sets of Lebesgue measure zero in sets of $\mathcal{H}^m$-measure zero, cf. again \cite{MGS1}) and Morrey's embedding imply $\mathcal{G}_u \subset \Gamma_u$ and  $\mathcal{H}^m (\Gamma_u \setminus \mathcal{G}_u) =0$. In particular $\llbracket \mathcal{G}_u\rrbracket = \llbracket \Gamma_u \rrbracket$. 

By \cite[Sec. 1.5, Th. 5]{MGS1}, the approximate tangent plane $T_{y}\mathcal{G}_{u}$ coincides for $\mathcal{H}^m$-a.e. $z_0 = (x_0, u (x_0)) \in \mathcal{G}_u$ or, with
\[
T_z\mathcal{G}_{u} = \{(x,Du(x_0)x): x \in \mathbb R^m\} \in \mathbb{G}_0 (m, m+n)\, .
\]
The following proposition allows then to pass from functionals defined on varifolds to classical functionals in the vectorial calculus of the variations (and viceversa). 

\begin{prop}\label{areaformula}
Let $u \in W^{1,m}(\Omega,\R^n)$, and define $v(x) := (x,u(x))$. Denote with $C_b(\Omega\times \R^n\times \mathbb{G}_0)$ the space of continuous and bounded functions on $\Omega\times \R^n\times \mathbb{G}_0$. Then, for every $\varphi \in C_b(\Omega\times \R^n\times \mathbb{G}_0)$, the following holds
\begin{equation}\label{areaweak}
\llbracket \mathcal{G}_{u}\rrbracket(\varphi) : = \int_{\mathcal{G}_u} \varphi(z,T_z\mathcal{G}_{u})d\mathcal{H}^m(z) = \int_{\Omega}\varphi(v(x),h(Du(x)))\mathcal{A}(Du(x))\dx.
\end{equation}
\end{prop}

 Consider therefore a functional
\[
\mathds{E}(u) := \int_{\Omega}f(Du(x))\dx\, ,
\]
for some $f: \R^{n\times m}\to \R$ with $$\frac{f(X)}{\mathcal{A}(X)} \in C_b(\R^{n\times m}).$$ Define moreover $F,G: \mathbb{G}_0 \to \R$ as
\[
F(M) := f(h^{-1}(M)),\; G(M):= \mathcal{A}(h^{-1}(M)).
\]
For any map $u \in W^{1,m}(\Omega,\R^n)$, we can apply \eqref{areaweak} to write:
\[
\int_{\Omega}f(D u(x))\dx= \int_{\Omega}F(h(D u(x)))\dx =  \int_{\Omega}\frac{F(h(D u(x)))}{G(h(Du(x)))}\mathcal{A}(Du(x))\dx =  \int_{\mathcal{G}_{u}}\Psi(T_z\mathcal{G}_{u})\di\mathcal{H}^m(z)\, ,
\]
where we have defined the map $\Psi$ on the open subset $\mathbb{G}_0$ of the Grassmanian  $\mathbb{G}(m,m + n)$ as $$\Psi (h(X)) := \frac{F(h(X))}{G(h(X))} = \frac{f(X)}{\mathcal{A}(X)}.$$
We are thus ready to introduce the following functional

\begin{Def}\label{d:functional_varifolds}
Let $V = \theta \llbracket \Gamma \rrbracket$ be an $m$-dimensional integer rectifiable varifold in $\R^{m+n}$ with the property that the approximate tangent $T_x \Gamma$ belongs to $\mathbb{G}_0$ for $\mathcal{H}^m$-a.e. $x\in \Gamma$. Then
\[
\Sigma(V) = \int_{\Gamma} \Psi(T_x\Gamma)\theta(x)d\mathcal{H}^m(x)\, .
\]
\end{Def}

The above discussion then proves the following

\begin{prop}
If $\Omega \subset \R^m$ and $u\in W^{1,m} (\Omega, \R^n)$, then $\Sigma (\llbracket \mathcal{G}_{u} \rrbracket)= \mathds E (u)$. Moreover, if $u\in W^{1,p} (\Omega, \R^n)$ with $p>m$, then $\Sigma (\llbracket \Gamma_u \rrbracket)= \mathds E (u)$.  
\end{prop}

\subsection{First variations}
We do not address here the issue of extending the functional $\Sigma$ to general varifolds (namely of extending $\Psi$ to all of $\mathbb{G}(m,m+n)$). Rather, assuming that such an extension exists, we wish to show that the usual stationarity of varifolds with respect to the functional $\Sigma$ is equivalent to stationarity with respect to two particular classes of deformations, which reduce to inner and outer variations in the case of graphs. We start recalling the usual stationarity condition. 

\begin{Def}\label{d:stationary}
Let $\Psi:\mathbb{G}(m,m+n) \to [0, \infty]$ be a continuous function. Fix a vector field $g \in C^1_c(\R^{m + n}; \R^{m + n})$ and define $X_\varepsilon$ as the flow generated by $g$, namely $X_\varepsilon(x) = \gamma_x(\varepsilon)$, if $\gamma_x$ is the solution of the following system
\[
\begin{cases}
\gamma'(t) = g(\gamma(t))\\
\gamma(0) = x.
\end{cases}
\]
We define the variation of $V$ with respect to the vector field $g \in C^1_c(\R^{m + n}; \R^{m + n})$ as
\[
[\delta_\Psi V] (g) := \lim_{\varepsilon \to 0}\frac{\Sigma((X_\varepsilon)_\#V) - \Sigma(V)}{\varepsilon}.
\]
$V$ is said to be \emph{stationary} if $[\delta_\Psi V] (g)=0, \forall g \in C^1_c(\R^{m + n}; \R^{m + n})$.
\end{Def}

Given an orthogonal projection $P\in \mathbb{G}_{m, m+n)}$, we introduce the notation $P^\perp$ for $\id_{m+n}-P$ (note that, if we identify $P$ with the linear space $L$ which is the image of $P$, then $P^\perp$ is the projection onto the orthogonal complement of $L$). 
From \cite[Lemma A.2]{DDG}, we know that, for $V = \theta\llbracket\Gamma\rrbracket$,
\begin{equation}\label{B}
[\delta_{\Psi}(V)](g) = \int_{\Gamma}\langle B_\Psi(T_x\Gamma),Dg(x)\rangle \theta(x)d\mathcal{H}^m(x), \forall g \in C^1_c(\R^{m + n},\R^{m + n})
\end{equation}
where $B_\Psi(\cdot) : \mathbb{G}(m,m + n) \to \R^{(m + n) \times (m + n)}$ is defined through the relation
\begin{equation}\label{BPsi}
\begin{split}
\langle B_\Psi(P), L\rangle = \Psi(P)\langle P, L \rangle + \langle d\Psi(P), P^\perp L P + (P^\perp L P)^T\rangle, 
\\
\forall P \in \mathbb{G}(m,m + n), \forall L \in \R^{(m + n)\times (m + n)},
\end{split}
\end{equation}
We are now ready to state our desired equivalence between stationarity of the map $u$ for the energy $\mathds{E}$ and stationarity of the varifold $\llbracket \mathcal{G}_u\rrbracket$ for the corresponding functional $\Sigma$. In what follows, given a function $g$ on $\mathcal{G}_u$ we will use the shorthand notation $\|g\|_{q, \mathcal{G}_u}$ for the norm $\|g\|_{L^q (\mathcal{H}^m \res \mathcal{G}_u)}$. 

\begin{prop}\label{link}
Fix any $m \le p \le +\infty$, $1\le q < +\infty$ and a Lipschitz, bounded, open set $\Omega \subset \R^m$. If a map $u\in W^{1,p}(\Omega,\mathbb{R}^n)$ satisfies
\begin{equation}\label{graphs}
\begin{cases}
\displaystyle \left|\int_{\Omega}\langle Df(D u),D v\rangle\dx\right| \le C\|v\mathcal{A}^{\frac{1}{q}}(D u)\|_q &\forall v\in C^1_c(\Omega,\mathbb{R}^n) \vspace{1mm}\\ 
\displaystyle\left|\int_{\Omega}\langle Df(D u), D u D \Phi(x)\rangle \dx - \int_{\Omega}f(D u)\dv(\Phi) \dx\right| \le C\|\Phi\mathcal{A}^{\frac{1}{q}}(D u)\|_q\;\; &\forall \Phi\in C_c^1(\Omega,\mathbb{R}^m),
\end{cases}
\end{equation}
for some $C \ge 0$, then the integer rectifiable varifold $\llbracket\mathcal{G}_{u}\rrbracket$ in $\mathbb{R}^{m + n}$ satisfies
\begin{equation}\label{varifolds}
\left|\delta_\Psi(\llbracket\mathcal{G}_{u} \rrbracket)(g)\right| \le C'\|g\|_{q, \mathcal{G}_u}\qquad \forall g\in C_c^1(\Omega\times\mathbb{R}^n,\mathbb{R}^{m + n}),
\end{equation}
for some number $C' = C'(C, m, p, q) \ge 0$. Conversely, if \eqref{varifolds} holds for some $C'$, then \eqref{graphs} holds for some $C = C (C',m,p,q)$. 
Moreover, $C'=0$ if and only $C=0$, namely $u$ is stationary for the energy $\mathds{E}$ if and only if $\llbracket\mathcal{G}_u\rrbracket$ is stationary for the energy $\Sigma$.
\end{prop}
\begin{rem}
As already noticed, when $p > m$ we can replace $\llbracket\mathcal{G}_{u}\rrbracket$ with $\llbracket \Gamma_u \rrbracket$. Moreover, under such stronger assumption, the proposition holds also for $q = \infty$, provided we set $\mathcal{A}(D u)^{\frac{1}{q}} := 1$ in that case.
\end{rem}

The proof of the previous proposition is a consequence of a few technical lemmas and will be given in the next section.

\section{Proof of Proposition \ref{link}}

 Let $f \in C^1(\R^{n \times m})$ be of the form $f(X) = \Psi(h(X))\mathcal{A}(X)$. In the next lemma we study the growth of the matrix-fields associated to the inner and the outer variations, i.e.
\begin{align}
A(X)&:= Df(X)\label{e:defA}\\
B(X)&:= f(X)\id_m- X^TDf(X).\label{e:defB}
\end{align}
Define also the matrix-field $V_f:\R^{n\times m} \to \R^{(m + n)\times (m + n)}$ to be
\begin{equation}\label{V}
V_f(X) := \frac{1}{\A(X)}
\left[\begin{array}{c|c}
B(X) & B(X)X^T \\ \hline
A(X) & A(X)X^T
\end{array}\right].
\end{equation}
In Lemma \ref{equivalence}, we will prove that
\[
B_\Psi(h(X)) = V_f(X),\; \forall X \in \R^{n \times m}.
\]
Combining Lemma \ref{growth} and \ref{equivalence} with the area formula we obtain Lemma \ref{equality},  from which we will infer Proposition \ref{link}.
\begin{lemma}\label{growth}
Let $\Psi \in C^1(\mathbb{G}(m,m+n))$ and let $f(X) = \Psi(h(X))\mathcal{A}(X)$, where $h$ is the map defined in $\eqref{hform}$. Then,
\begin{equation}\label{Fest}
\|A(X)\| \lesssim 1 +\|X\|^{\min\{m,n\} - 1}, \|B(X)\| \lesssim 1 + \|X\|^{\min\{n,m - 1\}}.
\end{equation}
\end{lemma}
In the statement of the Lemma and in the proof, the symbol $\Lambda \lesssim \Xi$ means that there exist a non-negative constant $C$ depending only on $n,m$ and on $\|\Psi\|_{C^1(\mathbb{G}(m,m+n)}$ such that
\[
\Lambda \le C\, \Xi\, .
\]
The lemma above is needed to get reach enough summability in order to justify the integral formulas in (the statement and the proof of) Lemma \ref{equality}. In some sense it is thus less crucial than the next lemma, which contains instead the core computations. For these reasons, the argument of Lemma \ref{growth}, which contains several lengthy computations is given in the appendix.

\begin{lemma}\label{equivalence}
For every $X \in \R^{n\times m}$,
\[
B_{\Psi}(h(X)) = V_f(X).
\]
\end{lemma}

\begin{lemma}\label{equality}
Let $f(X) = \Psi(h(X))\A(X)$ be a function of class $C^1(\R^{n\times m})$. Then, for every $g = (g^1,\dots, g^{m + n})\in C^1_c(\Omega\times\R^n)$, the following equality holds:
\begin{equation}\label{structure}
\delta_{\Psi}(\llbracket\mathcal{G}_{u}\rrbracket)(g) = \int_{\Omega}\langle B(Du(x)), D(g_1(x,u(x)))\rangle\dx +\int_{\Omega}\langle A(Du(x)), D(g_2(x,u(x)))\rangle\dx,
\end{equation}
where $g_1(x,y) := (g^1(x,y),\dots, g^m(x,y))$, $g_2(x,y) := (g^{m + 1}(x,y),\dots, g^{m + n}(x,y))$ and $A(X)$ and $B(X)$ are as in \eqref{e:defA} and \eqref{e:defB}.
\end{lemma}

We next prove Lemma \ref{equivalence} and Lemma \ref{equality} and hence end the section showing how to use Lemma \ref{equality} to conclude the desired Proposition \ref{link}.

\subsection{Proof of Lemma \ref{equivalence}}

For a map $g: \mathbb{G}(m+n,m) \to \R^\ell$, $\ell \ge1$, of class $C^1$, we denote the differential at the point $P \in \mathbb{G}(m+n,m)$ with the symbol $d_Pg$. Moreover, for $H \in T_P\mathbb{G}(m+n,m)$, and for $\gamma: (-1,1)\to \mathbb{G}(m+n,m)$ with $\gamma(0) = P$, $\gamma'(0) = H$, we denote
\[
d_Pg(P)[H]:= \lim_{t \to 0}\frac{g(\gamma(t))-g(P)}{t}.
\]
If $\ell = 1$, we identify $d_Pg(P)$ with the $\R^{(m+n) \times (m + n)}$ associated matrix representing the differential, and we denote $d_Pg(P)[H]$ with $\langle d_Pg(P),H\rangle$. In this proof, we will use the following facts:

\begin{itemize}
\item The tangent plane of $\mathbb{G}(m,m+n)$ at the point $P$ is given by
\[
T_P\mathbb{G}(m,m+n) = \{M\in \R^{(m+n)\times(m+n)}: M = P^\perp L P +(P^\perp L P)^T, \text{for some }L \in \R^{(m + n)\times(m+n)}\},
\]
as proved in \cite[Appendix A]{DDG}.
\item Let $h: \R^{n\times m} \to \mathbb{G}_0$ be the map defined in \eqref{hform}. Then, it is simple to verify that
\begin{equation}\label{hh}
h^{-1}(P) = P_2P_1^{-1}.
\end{equation}
Moreover, for every $ H \in T_P\mathbb{G}(m,m + n)$, one has:
\begin{equation}\label{diff1}
d_P(h^{-1})(P)[H] = (H_2 - P_2P_1^{-1}H_1)P_1^{-1} \in \R^{n\times m}.
\end{equation}

\item Recall that the area functional is defined as
\[
\A(X) = \sqrt{\det(M(X)^TM(X))} \qquad \mbox{where}\;\; M(X) =
\left[
\begin{array}{c}
\id_m\\
X
\end{array}
\right].
\]
Hence, for every $X,Y \in \R^{n \times m}$, we have
\begin{equation}\label{darea}
\langle D\mathcal{A}(X),Y\rangle= \frac{1}{2}\mathcal{A}(X)\tr[(M(X)^TM(X))^{-1}(Y^TX+X^TY)].
\end{equation}
\end{itemize}

Recall the definition of $B_{\Psi}(P)$ given in \eqref{BPsi}. Since $$\Psi(P) = \frac{f(h^{-1}(P))}{\A(h^{-1}(P))},$$ for every $ H \in T_P\mathbb{G}(m,m + n)$ we have

\begin{align*}
\langle d_P\Psi(P), H\rangle =\frac{1}{\mathcal{A}(X)}\langle Df(h^{-1}(P)), d_P(h^{-1})(P)[H]\rangle - \frac{f(h^{-1}(P))}{\mathcal{A}^2(h^{-1}(P))}\langle D\mathcal{A}(h^{-1}(P)),d_P(h^{-1})(P)[H]\rangle.
\end{align*}
When evaluated at $P = h(X)$, the previous expression reads
\begin{equation}\label{firsteq}
\langle d_P\Psi(h(X)), H\rangle =\frac{1}{\mathcal{A}(X)}\langle Df(X), d_P(h^{-1})(h(X))[H]\rangle - \frac{f(X)}{\mathcal{A}^2(X)}\langle D\mathcal{A}(X),d_P(h^{-1})(h(X))[H]\rangle.
\end{equation}

By \eqref{BPsi}, we know that, for every $L \in \R^{(m + n)\times (m + n)}$,
\begin{align*}
\langle B_{\Psi}(h(X)),L\rangle = \Psi(h(X))\langle h(X),L\rangle + \langle d_P\Psi(h(X)), h(X)^\perp L h(X) + (h(X)^\perp L h(X))^T\rangle.
\end{align*}
Therefore, we want to compute \eqref{firsteq} when
\[
H = h(X)^\perp L h(X) + (h(X)^\perp L h(X))^T = Lh(X) - h(X)Lh(X) + h(X)L^T - h(X)L^Th(X).
\]
We wish to find an expression for
\[
d_P(h^{-1})(h(X))[h(X)^\perp L h(X) + h(X) L^T h(X)^\perp]\, .
\] 
Using the decomposition introduced in \eqref{splitting} of $L$ in 4 submatrices, we compute
\begin{equation}
Lh(X) =
\left[\begin{array}{c|c}
L_1 & L_3 \\ \hline
L_2 & L_4
\end{array}\right]
\left[\begin{array}{c|c}
S & SX^T \\ \hline
XS & XSX^T
\end{array}\right]=
\left[\begin{array}{c|c}
L_1S + L_3XS & L_1SX^T + L_3XSX^T \\ \hline
L_2S + L_4XS & L_2SX^T + L_4XSX^T
\end{array}\right]
\end{equation}
and
\begin{equation}\label{buona}
\begin{split}
&h(X) L h(X) = \\
&\left[\begin{array}{c|c}
S(L_1 + L_3X + X^TL_2 + X^TL_4X)S & S(L_1 + L_3X + X^TL_2 + X^TL_4X)SX^T  \\ \hline
XS(L_1 + L_3X + X^TL_2 + X^TL_4X)S &XS(L_1 + L_3X + X^TL_2 + X^TL_4X)SX^T  
\end{array}\right]
\end{split}
\end{equation}

Combining \eqref{diff1} with \eqref{buona}, we get
\begin{equation}\label{d1}
\begin{split}
d_P(h^{-1})(h(X))[L h(X)] &= (L_2S + L_4XS - XSS^{-1}L_1S - XSS^{-1}L_3XS) S^{-1} \\
&= L_2 + L_4X - XL_1 - XL_3X,
\end{split}
\end{equation}
\begin{equation}\label{d2}
\begin{split}
d_P(h^{-1})(h(X))&[h(X) L h(X)] \\
&= XS(L_1 +L_3X+ X^TL_2 + X^TL_4X - S^{-1}SL_1 - S^{-1}SL_3X \\
& - S^{-1}SX^TL_2 - S^{-1}SX^TL_4X)SS^{-1}\\
&= XS(L_1 + L_3X + X^TL_2 + X^TL_4X - L_1 - L_3X - X^TL_2 - X^TL_4X) = 0
\end{split}
\end{equation}
and
\begin{equation}\label{d3}
\begin{split}
d_P(h^{-1})(h(X))[h(X) L^T]& = d_P(h^{-1})(h(X))[(L\circ h(X))^T]\\
&= (XSL_1^T + XSX^TL_3^T - XSL_1^T - XSX^TL_3^T)S^{-1} = 0.
\end{split}
\end{equation}
Combining \eqref{d1}, \eqref{d2} and \eqref{d3}, we get that
\begin{align*}
d_P(h^{-1})(h(X))[h(X)^\perp L h(X) + h(X) L^T h(X)^\perp] &= 
d_P(h^{-1})(h(X))(L h(X)) \\&= L_2 + L_4X - XL_1 - XL_3X.
\end{align*}
Now define the matrix:
\[
C := L_2 + L_4X - XL_1 - XL_3X.
\]
To expand \eqref{firsteq}, we now need to rewrite $$\langle D\mathcal{A}(X),d_P(h^{-1})(h(X))[H]\rangle.$$ First, we must compute the trace part coming from \eqref{darea}:
\begin{align*}
\tr[S(C^TX&+X^TC)] = \tr[S(L_2^TX + X^TL_4^TX - L_1^TX^TX - X^TL_3^TX^TX)] \\
& + \tr[S(X^TL_2 + X^TL_4X - X^TXL_1 - X^TXL_3X)]\\
&= 2\tr(SX^TL_2) + 2\tr(SX^TL_4X) - 2\tr(SX^TXL_1) - 2\tr(SX^TXL_3X).
\end{align*}
Hence, if $H =h(X)^\perp L h(X) + h(X)L^T h(X)^\perp$, we have just proved that:
\begin{equation}\label{sum1}
\begin{split}
\langle& d_{P}\Psi(h(X)), H\rangle = \frac{1}{\mathcal{A}(X)}\langle Df(X),  L_2 + L_4X - XL_1 - XL_3X \rangle \\
&  -\frac{f(X)}{\mathcal{A}(X)}(\tr(SX^TL_2) + \tr(SX^TL_4X) - \tr(SX^TXL_1) - \tr(SX^TXL_3X)).
\end{split}
\end{equation}
To conclude, we also need to compute 
\begin{equation}\label{sum2}
\begin{split}
\Psi(h(X))\langle h(X),L\rangle &= \frac{f(X)}{\mathcal{A}(X)}(\langle L_1, S\rangle + \langle L_2, XS\rangle + \langle L_3, SX^T\rangle + \langle L_4, XSX^T\rangle)\\
&= \frac{f(X)}{\mathcal{A}(X)}(\tr(SL_1) + \tr(SX^TL_2) + \tr(XSL_3) + \tr(XSX^TL_4)).
\end{split}
\end{equation}
Now we sum \eqref{sum1} and \eqref{sum2} to get $\langle B_{\Psi}(h(X)), L\rangle$. Using that $S^{-1}(X) = X^TX + \id_{m}$ and the invariance of the trace under cyclic permutations, we rewrite
\begin{align*}
&\tr(SL_1) + \tr(SX^TL_2) + \tr(XSL_3) + \tr(XSX^TL_4) \\
& \quad -\tr(SX^TL_2) - \tr(SX^TL_4X) + \tr(SX^TXL_1) + \tr(SX^TXL_3X) =
\tr(L_1) + \tr(L_3X).
\end{align*}
Combining our previous computations, we find
\begin{align*}
\langle B_\Psi(h(X)),L\rangle &= \frac{f(X)}{\mathcal{A}(X)}(\tr(L_1) + \tr(L_3X)) + \frac{1}{\mathcal{A}(X)}\langle Df(X),  L_2 + L_4X - XL_1 - XL_3X \rangle\\
&= \frac{1}{\mathcal{A}(X)}[- \langle X^TDf(X) + f(X)\id_m,L_1\rangle + \langle Df(X),L_2\rangle\\
& \qquad +\langle f(X)X^T - X^TDf(X)X^T,L_3\rangle  + \langle Df(X)X^T,L_4 \rangle].
\end{align*}
Since $L$ was arbitrary, we conclude that
\[
B_{\Psi}(h(X)) = \frac{1}{\mathcal{A}(X)}
\left[\begin{array}{c|c}
B(X) & B(X)X^T \\ \hline
A(X) & A(X)X^T
\end{array}\right].
\]

\subsection{Proof of Lemma \ref{equality}}
Fix $g$ as in the statement of the Lemma. By \eqref{B}, we know that
\[
\delta_{\Psi}(\llbracket\mathcal{G}_{u}\rrbracket)(g)  = \int_{\mathcal{G}_{u}}\langle B_{\Psi}(T_z\mathcal{G}_{u}),Dg(z)\rangle d\mathcal{H}^m(z).
\]
Now define $F(z,T):=\langle B_{\Psi}(T),Dg(z)\rangle$ and $\bar F(x,u(x)):= \langle B_{\Psi}(h(Du (x))),Dg(x,u(x))\rangle$. We have $F \in C_c(\Omega\times \R^n\times \mathbb{G}(m,m+n))$ and we apply Proposition \ref{areaformula} to find the equality
\begin{equation}\label{areaf}
\int_{\mathcal{G}_{u}}\langle B_{\Psi}(T_z\Gamma_{u}),Dg(z)\rangle d\mathcal{H}^m(z) = \int_{\Omega} \A(Du(x))\bar F(x,u(x))\dx.
\end{equation}
By Lemma \ref{equivalence},
\[
\bar F(x,u(x)) = \langle V_f(Du(x)),Dg(x,u(x)) \rangle 
\]
a.e. in $\Omega$.
Moreover, since
\[
\mathcal{A}(Du(x)) V_f(Du(x)) = \left[\begin{array}{c|c}
B(Du(x)) & B(Du (x))Du(x)^T \\ \hline
A(Du(x)) & A(Du(x))Du(x)^T
\end{array}\right],
\]
we have
\begin{align*}
\A(Du(x))\bar F(x,u(x)) &= \langle D_xg^1(x,u(x)), B(Du(x))\rangle + \langle B(Du(x))Du^T(x),  D_yg^1(x,u(x))) \rangle \\
&+ \langle D_xg^2(x,u(x)), A(Du(x))\rangle + \langle A(Du(x))Du^T(x),  D_yg^2(x,u(x))) \rangle\\
& = \langle B(Du(x)), D(g^1(x,u(x))) \rangle + \langle A(Du(x)), D(g^2(x,u(x))) \rangle.
\end{align*}
The previous equality and \eqref{areaf} yield the conclusion.

\subsection{Proof of Proposition \ref{link}}
First, assume \eqref{graphs}, and fix any $g \in C_c^1(\Omega\times \R^n,\R^{m+n })$, $g = (g^1,\dots, g^{m +n})$. Define 
\begin{align*}
\bar\Phi(x) &:= (g^1(x,u(x)),\dots, g^m(x,u(x))\\
\bar v(x)&:= (g^{m+1}(x,u(x)),\dots, g^{m + n}(x,u(x))\, .
\end{align*} 
We have $\bar\Phi \in L^\infty \cap W_0^{1,m}(\Omega,\R^m)$ and $\bar v \in L^\infty \cap W_0^{1,m}(\Omega,\R^n)$. Notice that we require \eqref{graphs} to hold only for $C^1$ maps with compact support, but Lemma \ref{growth} implies through an approximation argument that
\begin{equation}\label{extended}
\begin{cases}
\displaystyle \left|\int_{\Omega}A(D u),D v\rangle\dx\right| \le C\|v\A^{\frac{1}{q}}(Du)\|_q, \forall v\in L^\infty\cap W^{1,m}_0(\Omega,\mathbb{R}^n) \vspace{1mm}\\ 
\displaystyle\left|\int_{\Omega}B(D u),D \Phi\rangle \dx\right| \le C\|\Phi\A^{\frac{1}{q}}(Du)\|_q, \forall \Phi\in L^\infty\cap W^{1,m}_0(\Omega,\mathbb{R}^m).
\end{cases}
\end{equation}
Indeed, to prove, for instance, that the first inequality holds for any $v \in L^\infty\cap W^{1,m}_0$, pick a sequence $v_k \in C^\infty_c(\Omega,\R^n)$ such that $\|v_k\|_{L^\infty}$ is equibounded and $v_{k} \to v$ in $W^{1,m}$, $L^q$ and pointwise a.e.. The fact that
\[
\int_{\Omega}\langle A(D u),D v_k\rangle\dx \to \int_{\Omega}\langle A(D u),D v\rangle\dx
\]
is an easy consequence of the $W^{1,m}$ convergence of $v_k$ to $v$ and the fact that $A(Du) \in W^{\frac{m}{m - 1}}(\Omega,\R^{n\times m})$ by Lemma \ref{growth}. Moreover, the quantity
\[
\|v_k\A^{\frac{1}{q}}(Du)\|_q\to \|v\A^{\frac{1}{q}}(Du)\|_q
\]
by the dominated convergence theorem. Indeed, we required the pointwise convergence of $v_k$ to $v$ and moreover we can bound for every $k$ and almost every $x \in \Omega$:
\[
\|v_k\A^{\frac{1}{q}}(Du)\|^q(x) \le \sup_{k}\|v_k\|^q_{L^\infty}\mathcal{A}(Du(x)) \in L^1(\Omega). 
\]
Hence \eqref{extended} with $v_k$ instead of $v$ implies the same inequality for $v$ by taking the limit as $k \to \infty$. The proof of the second inequality of \eqref{extended} is analogous. We combine \eqref{extended} with \eqref{structure} to write
\[
|\delta_{\Psi}(\llbracket\mathcal{G}_{u}\rrbracket)(g)| \le \left|\int_{\Omega}\langle A(D u),D \bar v\rangle\dx\right| + \left|\int_{\Omega}\langle B(D u),D \bar \Phi\rangle \dx\right| \le C(\|\bar v\A^{\frac{1}{q}}(Du)\|_q + \|\bar \Phi\A^{\frac{1}{q}}(Du)\|_q).
\]
Notice that, since $\bar v(\cdot,u(\cdot)) \in L^\infty(\Omega,\R^n)$ and $\bar \Phi(\cdot,u(\cdot)) \in L^\infty(\Omega,\R^n)$, we have 
$$\|\bar v(\cdot,u(\cdot))\|^q\mathcal{A}(D u(\cdot)) + \|\bar \Phi(\cdot,u(\cdot))\|^q\A(D u(\cdot)) \in L^1(\Omega).$$ 
Now we use the trivial estimate $\|\bar v(x,y)\|\le \|g(x,y)\|$ for all $ x \in \Omega, y \in \R^n$, and area formula \eqref{areaweak} to conclude
\begin{align*}
\|\bar v\mathcal{A}^{\frac{1}{q}}(Du)\|^q_q &= \int_{\Omega}\|\bar v(x,u(x))\|^q\mathcal{A}(D u(x))\dx \le \int_{\Omega}\|g(x,u(x))\|^q\mathcal{A}(D u(x))\dx \\
&= \int_{\mathcal{G}_{u}}\|g\|^q(z)d\mathcal{H}^m(z) = \|g\|^q_{L^q (\mathcal{G}_{u})}.
\end{align*}
With analogous estimates, we also find
\[
\|\bar \Phi\mathcal{A}^{\frac{1}{q}}(Du)\|^q_q\le \|g\|^q_{L^q(\mathcal{G}_{u})}.
\]
Therefore, \eqref{varifolds} holds with constant $C' = 2C$. Now assume \eqref{varifolds}. Choose the following sequence $g_k \in C_c^1(\Omega \times \R^n)$:
\[
g_k(x,y) := G(x)\chi_k(y),
\]
where $G \in C_c^1(\Omega,\R^{n + m})$, and $\chi_k \in C^\infty_c(\R^n)$ with $0 \le \chi_k(y) \le 1, \forall y \in \R^n$, $\chi_k \equiv 1$ on $B_k(0)$, $\chi_k \equiv 0$ on $B^c_{2k}(0)$ and
$\|D \chi_k(y)\| \le \frac{1}{k}$,  for all $y \in \R^n$. Using again area formula \eqref{areaweak}, we write
\[
\|g_k\|^q_{L^q( \mathcal{G}_{u})} = \int_{\Omega}\|g_k(x,u(x))\|^q\A(D u(x))\dx.
\]
Monotone convergence theorem implies
\[
\lim_k\|g_k\|^q_{L^q(\mathcal{G}_{u})} = \|G\mathcal{A}^\frac{1}{q}(Du)\|^q_q.
\]
Now we want to use \eqref{structure}. Using the same notation as in the statement of Lemma \ref{equality}, i.e. splitting $G$ into $G_1 = (G^1,\dots, G^m)$ and $G_2 = (G^{n + 1},\dots, G^{n + m})$, we have
\begin{align*}
&\int_{\Omega}\langle B(D u(x)), D((g_k)_1(x,u(x)))\rangle\dx = \int_{\Omega}\langle B(D u(x)), D(\chi_k(u(x))G_1(x))\rangle\dx\\
&=\int_{\Omega}\chi_k(u(x))\langle B(D u(x)), DG_1(x)\rangle\dx + \int_{\Omega}\langle B(D u(x)),G_1(x)\otimes (D\chi_k(u(x))Du(x))\rangle\dx
\end{align*}
By Lemma \ref{growth} and the regularity of $G_1$, we have that
\begin{equation}\label{domin}
\|DG_1\|\|B(Du)\| \in L^1(\Omega) \text{ and } \|G_1\|\|B(Du)\|\|Du\| \in L^1(\Omega).
\end{equation}
Since
\[
\chi_k(u(x))\langle B(D u(x)), DG_1(x)\rangle \to \langle B(D u(x)), DG_1(x)\rangle
\]
pointwise a.e. as $k \to \infty$, \eqref{domin} tells us that we can apply dominated convergence theorem to infer
\[
\lim_{k\to \infty}\int_{\Omega}\chi_k(u(x))\langle B(D u(x)), DG_1(x)\rangle\dx = \int_{\Omega}\langle B(D u(x)), DG_1(x)\rangle.
\]
Moreover using the pointwise bound $\|D\chi_k(u(x))\| \le \frac{1}{k}$,
\[
\left| \int_{\Omega}\langle B(D u(x)),G_1(x)\otimes (D\chi_k(x)Du(x))\rangle\dx\right| \le  \frac{1}{k}\int_{\Omega}\|B(D u(x))\|\|G_1(x)\|Du(x)\|\dx.
\]
Again through \eqref{domin}, we infer that the last term converges to $0$. This implies that
\[
\int_{\Omega}\langle B(D u(x)), D((g_k)_1(x,u(x)))\rangle\dx \to \int_{\Omega}\langle B(D u(x)), DG_1(x)\rangle\dx \text{ as }k\to \infty.
\]
In a completely analogous way,
\[
\int_{\Omega}\langle A(D u(x)), D((g_k)_2(x,u(x)))\rangle\dx \to \int_{\Omega}\langle A(D u(x)), DG_2(x)\rangle\dx \text{ as }k\to \infty.
\]
Now \eqref{structure} and the previous computations yield
\begin{align*}
&\int_{\Omega}\langle A(D u(x)), DG_2(x)\rangle\dx + \int_{\Omega}\langle B(D u(x)), DG_1(x)\rangle\dx \\
&=\lim_{k \to \infty}\left[\int_{\Omega}\langle A(D u(x)), D(g_k)_2(x)\rangle\dx + \int_{\Omega}\langle B(D u(x)), D(g_k)_1(x)\rangle\right]\dx\\
&= \lim_{k\to \infty}\delta_\Psi(\llbracket\mathcal{G}_{u}\rrbracket)(g_k) \le C'\lim_{k}\|g_k\|_{L^q( \mathcal{G}_{u})} = C'\|G\mathcal{A}^{\frac{1}{q}}(Du)\|_q ,
\end{align*}
and it is immediate to see that this implies \eqref{graphs} with constant $\bar C' = C'$.

\section{Some open questions}\label{questions}

We list here a series of questions related to the topic of the present paper. Firstly, as already explained in the introduction, the main question which motivated the investigations of this paper is the following widely open question. 

\begin{Q}\label{Q1}
Is it possible to prove an analog of W. Allard's celebrated regularity theorem \cite{ALLARD} if we consider strongly elliptic integrands (in the sense of Almgren) $\Psi$ on Grassmanian?
\end{Q}

The answer to this question is far from being immediate. A major obstacle is the lack of the monotonicity formula,   \cite{ALLMON}. Actually most of the proof in \cite{ALLARD} can be carried over if one know the validity of a Michael-Simon inequality. More precisely, consider a rectifiable varifold $V =  \theta\llbracket\Gamma\rrbracket$ with density bounded below (e.g. \(\theta \ge 1\)) and anisotropic variation $\delta_\Psi  V$ which is bounded in $L^1 (\theta \mathcal{H}^m \res \Gamma)$, i.e.
\[
\delta_\Psi (g) = \int _\Gamma H_\Psi \cdot g\, \theta\, d\mathcal{H}^m 
\]
for some $H_\Psi \in L^1$. The anisotropic Michael-Simon inequality would then take the conjectural form 
\begin{equation}\label{e:MS}
\left(\int_\Gamma h^{\frac{m}{m-1}} \theta\, d\mathcal{H}^m \right)^{\frac{m-1}{m}} \le C \int_\Gamma \theta |\nabla^\Gamma h| + C \int_\Gamma |h|\, |H_\Psi|\, \theta\,  d \mathcal{H}^m\,
\qquad \forall h \in C^1_c \, .
\end{equation}

\begin{Q}\label{Q2}
Is it possible to prove a Michael-Simon inequality as \eqref{e:MS}  for (at least some) anisotropic energies?
\end{Q}

Of course, Question \ref{Q1} has its counterpart on graphs, which amounts to extend the partial regularity of Evans for minimizers to stationary points. 

\begin{Q}\label{Q3}
Is it possible to extend the partial regularity theorem of \cite{EVA} to Lipschitz graphs that are stationary with respect to strongly polyconvex (or quasiconvex) energies?
\end{Q}

Answering these questions in this generality seems out of reach at the moment. It is however possible to formulate several interesting intermediate questions, many of which are related to the ``differential inclusions point of view'' adopted in the present paper. 

First of all we could consider stronger assumptions on the integrand $\Psi$. 
In the recent paper \cite{DDG}, A. De Rosa, the second named author and F. Ghiraldin introduce the so-called Atomic Condition. Such condition characterizes those energies for which (the appropriate extension of) Allard's rectifiability result holds. The following question is thus natural (see the forthcoming paper \cite{AT} for results in this direction):

\begin{Q}\label{Qadd1}
What is the counterpart of the Atomic Condition for functionals on graphs and what can be concluded from it in the graphical case?
\end{Q}

Secondly, a possible approach to Question \ref{Q1} is a continuation-type argument on the space of all energies. Since the area functional has a particular status, the following question is particularly relevant.

\begin{Q}\label{Qadd2}
Does an Allard type result holds for integrands which are sufficiently close to the area?
\end{Q}

In the forthcoming paper, \cite{AREA}, the fourth named author proves a partial result in the above direction. Using methods coming from the theory of differential inclusion, \cite{AREA} shows that graphs with small Lipschitz constant that are stationary with respect to functions sufficiently close to the area are regular. These results, other than the one in the present paper, seem to point to partial regularity for stationary varifolds (or graphs), as opposed to the situation of \cite{SMVS,LSP}.  

We note that a key step in the proof of Evans' partial regularity theorem is the so called Caccioppoli inequality which, roughly, reads as follows: for a minimizer \(u\) defined on \(B_2\)
\[
\int_{B_1} |Du-Da|^{2}\le C \int_{B_2} |u-a|^2 
\] 
for all affine functions \(a(x)=b+A x\). The geometric counterpart of this estimate  is used by Almgren in its partial regularity theorem for currents minimizing  anisotropic energies, \cite{Almgren68}. These inequalities  are obtained by direct comparison with suitable competitors. A similar estimate is obtained, by purely PDE techniques, by Allard in the case of stationary varifolds and it is one of the key step in establishing his regularity theorem for stationary varifolds. For co-dimension one stationary varifolds which are stationary with respect to anisotropic \emph{convex} integrands, a similar inequality is known to hold true, \cite{Allard86}. However, in general co-dimension, no condition on the integrand it is known to ensure its validity, not even in  neighbourhoods of the area integrand, \(\Psi=1\).

\begin{Q}
Which conditions on the integrands \(f\) or \(\Psi\) ensures the validity of a Cacciopoli type inequality for stationary points? 
\end{Q}

\medskip

In \cite{LASLOC}, it is proved that for differential inclusions of the form
\[
D u \in K,
\]
where $K$ is a compact set of $\R^{2\times 2}$ that does not contain $T_4$ configurations, compactness properties hold. In particular, if $\sup_{j}\|u_j\|_{W^{1,p}(B_1(0))}<+\infty$ for some $p > 1$, then there exists a subsequence $u_{j_k}$ such that $u_{j_k}$ converges strongly in $W^{1,q}(B_1(0))$ for every $q < p$. This kind of compactness property can actually be used to prove partial regularity of solutions to elliptic systems of PDEs. The strategy of \cite{LASLOC} does not apply directly to the higher dimensional case and motivates the following question

\begin{Q}
Let $f\in C^1 (\R^{n\times m})$ be a strictly polyconvex function and $K_f\subset \R^{(2n+m)\times m}$. Suppose $W_j : \Omega \to \R^{(2n + m) \times m}$ is a sequence of maps such that $\sup_j \|W_j\|_{\infty} < + \infty$ and
\[
\dist(W_j(x), K_f) \rightharpoonup 0
\]
in the weak topology of $L^p$. Then, up to subsequences, does $W_j$ converge strongly in $W^{1,p}$?
\end{Q}

 To formulate the next questions, let us recall the following definitions, see, for instance, \cite{MSK}, or \cite[Section 4.4]{DMU}. A function
\[
F: \R^{n\times m} \to \R
\]
is said to be rank-one convex if $h(t) := F(X + tY)$ is a convex function, for every $X, Y \in \R^{n\times m}$ and $\rank(Y) = 1$. It is said to be quasiconvex if $\forall \Phi \in C^{\infty}_c(B_1(0),\R^n), B_1(0) \subset \R^m$ and $X \in \R^{n\times m}$, one has
\[
\fint_{B_1(0)} F(X + D\Phi(x))\dx \ge F(X).
\]
For compact sets $K \subset \R^{n\times m}$, we define
\[
K^{rc} = \{X \in K: F(X) \le 0, \forall F:\R^{n\times m}\to \R \text{ rank-one convex s.t. } F(Z) \le 0, \forall Z \in K\}
\]
and analogously $K^{qc}$ (resp. $K^{pc}$) where one uses quasiconvex (resp. polyconvex) functions instead of rank-one convex functions. Moveover, one has the following chain of inclusions
\begin{equation}\label{chain}
K \subseteq K^{rc} \subseteq K^{qc} \subseteq K^{pc}.
\end{equation}
A necessary condition for compactness to hold is that $K^{qc} = K$ so that in particular $K^{rc} = K$. These results are consequence of the theory of Young Measures and in particular of the abstract result of \cite{KPE}. For a thorough explanation, we refer once again the reader to \cite[Section 4]{DMU}.

As discussed at the beginning of Section \ref{TpN}, in dimension 2 the stationarity of a graph is equivalent to solve
\begin{equation}\label{2dim}
D W \in K_f := \left\{A \in \R^{(2n + 2)\times 2}:
A =
\left(
\begin{array}{c}
X\\
Df(X)J\\
X^TDf(X)J - f(X)J
\end{array}
\right)
\right\},
\end{equation}
for some $W\in W^{1,\infty}(\Omega,\R^{2n + 2})$, where $J$ is the symplectic matrix
\[
\left( 
\begin{array}{ll}
0 & 1\\
-1 & 0
\end{array}\right)\, .
\]
Therefore we can ask
\begin{Q}\label{2dimq}
Let $f\in C^1 (\R^{n\times 2})$ be a strictly polyconvex function.
Is it true that $(K_f\cap \bar{B}_R(0))^{rc} = K_f\cap \bar{B}_R(0)$ for every $R > 0$?
\end{Q}

The same question can be generalized to $m > 2$ using the wave cone $\Lambda_{dc}$ of Definition \ref{d:cone_dc}. In analogy with rank-one convex functions, we can introduce $\Lambda_{dc}$-convex functions

\begin{Def} A function $F:\R^{(2n +m)\times m} \to \R$ is $\Lambda_{dc}$-convex if $h(t) := F(X + tY)$ is a convex function for every $X\in \R^{(2n+m)\times m}$ and $Y \in \Lambda_{dc}$. We also define, for a compact set $K \subset \R^{(2n+m) \times m}$ the $\Lambda_{dc}$-convex hull
\[
K^{dc} := \{X \in K: F(X) \le 0, \forall F:\R^{(2n+m_\times m)} \Lambda_{dc}\text{-convex s.t. } F(Z) \le 0, \forall Z \in K\}.
\]
\end{Def}
The multi-dimensional analogue of Question \ref{2dimq} is then the following:
\begin{Q}
Let $f\in C^1 (\R^{n\times m})$ be a strictly polyconvex function and $R>0$. Does 
\[
(K_f\cap \bar{B}_R(0))^{dc} = K_f\cap \bar{B}_R(0) ?
\]
\end{Q}

\appendix

\section{Proof of Proposition \ref{areaformula}}

First, by \cite[Sec. 1.5,Th. 1]{MGS1}, one has that if $w \in W^{1,m}(\Omega, \R^{m + n})$, then for every measurable set $A \subset \Omega$ and every measurable function $g: \R^{m + n} \to \R$ for which
\begin{equation}\label{request}
g(w(\cdot))J_w(\cdot) \in L^1(A),
\end{equation}
it holds
\[
\int_{A}g(w(x))J_w(x)\dx = \int_{\R^{m + n}}g(z)N(w,A,z)d\mathcal{H}^m(z),
\]
where
\[
J_w(x) = \sqrt{\det(Dw(x)^TDw(x))} \text{ and } N(w,A,z) := \#\{x: x \in A\cap A_D(w), w(x) = z\}.
\]
We want to apply this result with
\[
A = \mathcal{L}_u, \; w(x) = v(x), \; g(x,y) := F(v(x),T_{v(x)}\mathcal{G}_{u}), \forall x \in \Omega, y \in \R^m.
\]
In this way, it is straightforward by the fact that $\mathcal{R}_u = \mathcal{L}_u\cap A_D(u)$ and the definition of $v(x)$ that $N(v,\mathcal{L}_u,z) = 1$ for $\mathcal{H}^m\res \mathcal{G}_{u}$ and $N(v,\mathcal{L}_u,z) = 0$ if $z \notin \mathcal{G}_{u}$. Hence:
\[
\int_{\R^{m + n}}g(z)N(w,A,z)d\mathcal{H}^m(z) = \int_{\mathcal{G}_{u}}F(v(x),T_{v(x)}\mathcal{G}_{u})d\mathcal{H}^m(z).
\]
Moreover, since $|\Omega\setminus \mathcal{L}_u| = 0$ and $J_w(x) = \mathcal{A}(Du(x))$, we also find
\[
\int_{\mathcal{L}_u}g(w(x))J_w(x)\dx = \int_{\Omega}F(v(x),T_{v(x)}\mathcal{G}_{u})\mathcal{A}(Du(x))\dx.
\]
Since $u \in W^{1,m}(\Omega,\R^n)$ and $F \in C_b(\Omega\times\R^n\times\mathbb{G}_0)$, \eqref{request} is fulfilled and we can apply the aforementioned result to obtain the desired equality \eqref{areaweak}.

\section{Proof of Lemma \ref{growth}}
First of all we compute $D\mathcal{A}(X)$. Recall the notation on multi-indices introduced in Definition \ref{multiind} and the definition of the matrix $\overline{{\cof}(X^Z)^T}$ in the proof of Proposition \ref{p:convexity}. Then, since
\[
\mathcal{A}(X) = \sqrt{1 +\|X\|^2 + \sum_{2\le r \le \min\{m,n\}}\sum_{Z \in \mathcal{A}_r}\det(X_Z)^2},
\]
we have
\begin{equation}\label{diffarea}
D\mathcal{A}(X) = \frac{X + \sum_{2\le r \le \min\{m,n\}}\sum_{Z \in \mathcal{A}_r}\det(X_Z)\overline{{\cof}(X^Z)^T}}{\mathcal{A}(X)},\; \forall X \in \R^{n\times m}.
\end{equation}
Next, we observe that by the chain rule
\[
D(\Psi(h(X))_{ij} = \sum_{1\le \alpha,\beta\le m + n}(\partial_{\alpha\beta}\Psi)(h(X))\partial_{ij}h_{\alpha\beta}(X),
\]
hence
\begin{equation}\label{chain}
D(\Psi(h(X)) =\sum_{1\le \alpha,\beta\le m + n}(\partial_{\alpha\beta}\Psi)(h(X))Dh_{\alpha\beta}(X).
\end{equation}
We can therefore write
\begin{align*}
A(X) &= \Psi(h(X))D\mathcal{A}(X) + \mathcal{A}(X)D(\Psi(h(X)) \\
&= \Psi(h(X))D\mathcal{A}(X) + \mathcal{A}(X)\sum_{1\le \alpha,\beta\le m + n}(\partial_{\alpha\beta}\Psi)(h(X))Dh_{\alpha\beta}(X)
\end{align*}
and
\[
B(X) = \Psi(h(X))(-X^TD\mathcal{A}(X) +\mathcal{A}(X)\id_m) + \mathcal{A}(X)\sum_{1\le \alpha,\beta\le m + n}(\partial_{\alpha\beta}\Psi)(h(X))X^TDh_{\alpha\beta}(X).
\]
Since $\mathbb{G}(m,m + n)$ is compact, we have that both $\Psi(h(X))$ and $(D\Psi)(h(X))$ are bounded in $L^\infty(\R^{n\times m})$ by a constant $c > 0$ and using \eqref{diffarea}, we can bound
\[
\Psi(h(X))\|D\mathcal{A}(X)\| \lesssim \|X\|^{\min\{m,n\} - 1}.
\]
Moreover, for every $X \in \R^{n \times m}$, $2 \le r \le \min\{m,n\}$ and $Z \in \mathcal{A}_r$, we have
\[
X^T\overline{{\cof}(X^Z)^T} = \det(X_Z)I_Z,
\]
where, if $\delta_{ab}$ denotes Kronecker's delta and $Z$ has the form $Z = (i_1,\dots, i_r,j_1,\dots, j_r)$, $I_Z$ is the $m\times m$ matrix defined as
\[
(I_Z)_{ij}=
\begin{cases}
0, &\text{ if $i \neq i_a$} \text{ or } j \neq j_b, \forall a,b,\\
\delta_{ab}, &\text{ if }i = i_a, j=j_b.
\end{cases}
\]
Therefore
\begin{equation}\label{almost}
-X^TD\mathcal{A}(X) + \mathcal{A}(X)\id_m = -\frac{X^TX + \sum_{2\le r \le \min\{m,n\}}\sum_{Z \in \mathcal{A}_r}\det^2(X_Z)I_Z - \mathcal{A}^2(X)\id_{m}}{\mathcal{A}(X)}.
\end{equation}
If $n \le m - 1$, then the best way to estimate the previous expression is simply
\[
\|X^TD\mathcal{A}(X) - \mathcal{A}(X)\id_m\| \lesssim 1 + \|X\|^n.
\]
On the other hand, if $n \ge m$, then for $Z \in \mathcal{A}_m$ we have $I_Z = \id_m$, hence \eqref{almost} becomes
\begin{align*}
-X^TD\mathcal{A}(X) + \mathcal{A}(X)\id_m &= -\frac{X^TX + \sum_{2\le r \le \min\{m,n\}}\sum_{Z \in \mathcal{A}_r}\det^2(X_Z)I_Z - \mathcal{A}^2(X)\id_{m}}{\mathcal{A}(X)}\\
& = - \frac{X^TX - (1 + \|X\|^2)\id_m + \sum_{2\le r \le m - 1}\sum_{Z \in \mathcal{A}_r}\det^2(X_Z)(I_Z-\id_{m})}{\mathcal{A}(X)}.
\end{align*}
In this case
\[
\|X^TD\mathcal{A}(X) - \mathcal{A}(X)\id_m\| \lesssim 1 + \|X\|^{m - 1}.
\]
To conclude the proof of the Lemma, we still need to prove that for every $1\le i,j \le m + n$,
\begin{equation}\label{boilsdown}
\mathcal{A}(X)\|Dh_{ij}(X)\| \lesssim 1 + \|X\|^{\min\{n,m\} - 1}, \; \mathcal{A}(X)\|X^TDh_{ij}(X)\| \lesssim 1 + \|X\|^{\min\{m-1,n\}}.
\end{equation}
To perform the computation, we need to divide it into cases corresponding to the four blocks of the matrix $h(X)$ as written in \eqref{hform}. To this end, recall the notation
\[
S(X) = (\id_m + X^TX)^{-1},
\]
and moreover notice that $h(X)$ is symmetric, therefore we just need to prove \eqref{boilsdown} in the case $i \le j$. Another useful fact is the following. First notice that for every matrices $N \in O(n)$, $M \in O(m)$ ($O(k)$ is the group of orthogonal matrices of order $k$), one has
\[
S(NXM) = M^TS(X)M.
\]
From an easy computation we then conclude that, for every $1 \le i,j \le m + n$ and for every $X \in \R^{n\times m}, N \in O(n), M \in O(m)$,
\begin{align}
\|Dh_{ij}(X)\| &\lesssim \sum_{1\le a,b \le m + n} \|Dh_{ab}(NXM)\|\label{simp}\\
\|X^TDh_{ij}(X)\| &\lesssim \sum_{1\le a,b \le m + n} \|(NXM)^TDh_{ab}(NXM)\|\label{simp2}.
\end{align}
Since also $\mathcal{A}(X) = \mathcal{A}(NXM)$, $\forall X\in \R^{n\times m}, M \in O(m), N \in O(n)$, \eqref{simp}-\eqref{simp2} tell us that we can check estimates \eqref{boilsdown} just on matrices $Y:=NXM$ with two additional hypotheses. Fix $X \in \R^{n\times m}$, define $Z = XM$ and denote the $j$-th column of a matrix $A \in \R^{n\times m}$ with $A^j$. First, by a suitable choice of $M$, we can make sure that $Y^TY = Z^TZ = M^TX^TXM$ is diagonal. Once this choice is made, if $n \ge m$, then we choose $N = \id_n$. Otherwise, if $n < m$, then we observe that at most $n$ of the columns of $Z$ are non-zero, let these be $Z^{j_1},\dots, Z^{j_n}$ and let us define $J:= \{j_1,\dots, j_n\}$ with $1 \le j_1 < j_2<\dots < j_n \le m $. If for some $j_k$ we have $Z^{j_k} = 0$, then we set $N = \id_n$. Otherwise, the $n \times n$ matrix $V$ formed using $Z^{j_1},\dots,Z^{j_n}$ has columns that are pairwise orthogonal and nonzero, hence there exists $O \in O(n)$ such that
\[
V = OD,
\]
with $D$ diagonal. In this case, we choose $N = O^T$, so that the resulting $Y$ has the property that
\[
Y^j =
\begin{cases}
y_\ell e_\ell &\text{ if } j = j_\ell, j_\ell \in \{j_1,\dots, j_n\}, \\
0 &\text{ otherwise},
\end{cases}
\]
where $y_j \in \R$ and $e_\ell$ are the vectors of the canonical basis of $\R^n$. Notice that this choice of $M$ and $N$ also implies that
\[
\mathcal{A}(Y) = \sqrt{\prod_{i = 1}^m(1 + \|Y^i\|^2)} \text{ and } S(Y) = \diag((1 + \|Y^1\|^2)^{-1},\dots,(1 + \|Y^m\|^2)^{-1}).
\]
We call (HP) these assumptions on the matrix $Y \in \R^{n\times m}$.
\\
\\
\fbox{First case, $1\le i\le j \le m$:} In this case, $h_{ij} = S_{ij}$. We have
\[
\sum_{1\le k \le m}S^{-1}_{ik}S_{kj} = \delta_{ij},
\]
hence, taking a derivative,
\[
\sum_{1\le k \le m}\partial_{ab}S^{-1}_{ik}S_{kj} + \sum_{1\le k\le m}S^{-1}_{ik}\partial_{ab}S_{jk} = 0.
\]
We can invert the previous relation to get
\begin{equation}\label{inverse}
\partial_{ab}S_{kl} = -\sum_{1\le c,d \le m}S_{kc}S_{ld}\partial_{ab}S^{-1}_{cd} .
\end{equation}
Finally, since $S^{-1}_{ik} = \delta_{ik} + \sum_{1\le l \le m}x_{li}x_{lk}$, we have
\[
\partial_{ab}S^{-1}_{ik} = \sum_{1\le c \le m}\delta_{ci}^{ab}x_{ck} + \sum_{1\le c \le m}\delta_{ck}^{ab}x_{ci},
\]
where the symbol $\delta_{\alpha\beta}^{cd} = 0$ if $\alpha \neq c$ or $\beta \neq d$, otherwise $\delta_{\alpha\beta}^{cd} = \delta_{\alpha\beta}^{\alpha\beta} = 1$. We can therefore use \eqref{inverse} to write
\begin{equation}\label{partialfinal}
\begin{split}
\partial_{ab}S_{ij} &= -\sum_{1\le k,l \le m}S_{ik}S_{jl}\left( \sum_{1\le c \le m}\delta_{ck}^{ab}x_{cl} + \sum_{1\le c \le m}\delta_{cl}^{ab}x_{ck}\right) \\
&= -\sum_{1\le k,l,c \le m}S_{ik}S_{jl}\delta_{ck}^{ab}x_{cl} - \sum_{1\le k,l,c \le m}\delta_{cl}^{ab}x_{ck}S_{ik}S_{jl} = -\sum_{1\le l \le m}\left(S_{ib}S_{jl}x_{al} + x_{al}S_{il}S_{jb}\right).
\end{split}
\end{equation}
Moreover,
\[
(X^TDS_{ij})_{cd} = \sum_{1 \le a \le n }x_{ac}\partial_{ad}S_{ij} =-\sum_{1 \le l \le m, 1\le a\le n}\left(S_{id}S_{jl}x_{al}x_{ad} + x_{ad}x_{al}S_{il}S_{jd}\right).
\]
Now we use our previous observation \eqref{simp}-\eqref{simp2} to consider $Y$ satisfying (HP), so that in particular $Y^TY$ is diagonal. In this case, we have
\begin{align*}
|\partial_{ab}S_{ij}(Y)| \le \sum_{1\le l \le m}\left(|S_{ib}S_{jl}y_{al}| + |y_{al}S_{il}S_{jb}|\right).
\end{align*}
For every $1\le i,b,j,l \le m, 1\le a \le n$,
\[
\mathcal{A}(Y)|S_{ib}S_{jl}y_{al}| \le \sqrt{\prod_{c = 1}^{m}(1 + \|Y^i\|^2)}\frac{|y_{al}|}{(1 + \|Y^b\|^2)(1 + \|Y^l\|^2)}.
\]
Let us explain in detail how to get the desired estimate \eqref{boilsdown} in this case. Notice that either $Y^l$ is $0$, and in this case there is nothing to prove, or $Y^l \neq 0$. Thanks to (HP), in $Y$ there are at most $\min\{m,n\}$ non-zero columns. First let $m \le n$, then:
\[
\sqrt{\prod_{c = 1}^{m}(1 + \|Y^c\|^2)}\frac{|y_{al}|}{(1 + \|Y^b\|^2)(1 + \|Y^l\|^2)} \lesssim \sqrt{\prod_{c = 1}^{m}(1 + \|Y^c\|^2)}\frac{1}{\sqrt{1 + \|Y^l\|^2}} \lesssim 1 + \|Y\|^{m - 1}.
\]
If $n < m$ and $J$ is the set on indices corresponding to non-zero columns, we are in the hypothesis in which $l \in J$. Therefore we have
\[
\sqrt{\prod_{c = 1}^{m}(1 + \|Y^c\|^2)}\frac{|y_{al}|}{(1 + \|Y^b\|^2)(1 + \|Y^l\|^2)} \lesssim \sqrt{\prod_{c \in J}(1 + \|Y^c\|^2)}\frac{1}{\sqrt{1 + \|Y^l\|^2}} \lesssim 1 + \|Y\|^{m - 1}.
\]
This proves that 
\begin{equation}\label{ff}
\|Dh_{ij}(Y)\| \lesssim 1 + \|Y\|^{\min\{m,n\} - 1} \text{ for } 1\le i,j \le m.
\end{equation}
 We also have
\[
\mathcal{A}(Y)|(Y^TDS_{ij})_{cd}(Y)| \le \mathcal{A}(Y)\sum_{1 \le l \le m, 1\le a\le n}\left(|S_{id}S_{jl}y_{al}y_{ad}| + |y_{ad}y_{al}S_{il}S_{jd}|\right).
\]
Analogously to the previous case, we estimate for every $1\le i,d,j,l\le m, 1\le a \le n$,
\[
\mathcal{A}(Y)|S_{id}S_{jl}y_{al}y_{ad}| \le \sqrt{\prod_{c = 1}^m(1 + \|Y^c\|^2)}\frac{|y_{al}||y_{ad}|}{(1 + \|Y^d\|^2)(1 + \|Y^l\|^2)},
\]
and the desired estimate is obtained with a reasoning completely analogous to the one of \eqref{ff}. This concludes the proof of this case.
\\
\\
\fbox{Second case, $1\le i \le m < m+1\le j \le m+n$:} From now on we use $m+j$ rather than $j$ for the corresponding index. We thus have 
\[
h_{ij+m}(X) = (S(X)X^T)_{ij} = \sum_{k = 1}^mS_{ik}x_{jk}\, .
\] 
We compute the derivative using \eqref{partialfinal}:
\[
\partial_{ab}h_{ij+m}(X) = \sum_{k = 1}^m\delta_{jk}^{ab}S_{ik} +  \sum_{k = 1}^mx_{jk}\partial_{ab}S_{ik} = \sum_{k = 1}^m\delta_{jk}^{ab}S_{ik} -\sum_{1\le l,k \le m}\left(S_{ib}S_{kl}x_{al}x_{jk} + x_{al}x_{jk}S_{il}S_{kb}\right),
\]
and also
\begin{align*}
(X^TDh_{ij+m}(X))_{ab}= \sum_{1\le c \le n}x_{ca}\partial_{cb}h_{ij} &= \sum_{1\le k \le m, 1 \le c \le n}x_{ca}\delta_{jk}^{cb}S_{ik}\\
&- \sum_{1\le l,k \le m, 1 \le c \le n}\left(x_{ca}S_{ib}S_{kl}x_{cl}x_{jk} + x_{ca}x_{cl}x_{jk}S_{il}S_{kb}\right)\\
&= x_{ja}S_{ib} -\sum_{1\le l,k \le m, 1 \le c \le n}\left(x_{ca}S_{ib}S_{kl}x_{cl}x_{jk} + x_{ca}x_{cl}x_{jk}S_{il}S_{kb}\right)
\end{align*}
Since $S^{-1}(X) = \id_m + X^TX$,
\begin{equation}\label{def}
\delta_{ij} =\sum_{1\le k \le m} S_{ik}(\delta_{kj} + \sum_{1\le c \le n}x_{ck}x_{cj}) =S_{ij} + \sum_{1\le k \le m,1\le c \le n}S_{ik}x_{ck}x_{cj},
\end{equation}
hence we can rewrite
\begin{equation}\label{secondsecond}
\begin{split}
(X^TDh_{ij+m}(X))_{ab}&=  x_{ja}S_{ib} -\sum_{1\le l,k \le m, 1 \le c \le n}\left(x_{ca}S_{ib}S_{kl}x_{cl}x_{jk} + x_{ca}x_{cl}x_{jk}S_{il}S_{kb}\right)\\
& =  x_{ja}S_{ib} -\sum_{k =1}^m\left(S_{ib}x_{jk}\sum_{1\le l \le m, 1 \le c \le n}x_{ca}S_{kl}x_{cl}+ x_{jk}S_{kb}\sum_{1\le l \le m, 1 \le c \le n}x_{ca}x_{cl}S_{il}\right)\\
& = x_{ja}S_{ib} -\sum_{k =1}^m S_{ib}x_{jk}(\delta_{ka} - S_{ka})-\sum_{k = 1}^mx_{jk}S_{kb}(\delta_{ai} - S_{ai}) \\
& = \sum_{k =1}^m S_{ib}x_{jk}S_{ka} +\sum_{k = 1}^mx_{jk}S_{kb}\delta_{ai} + \sum_{k = 1}^mx_{jk}S_{kb}S_{ai}.
\end{split}
\end{equation}
Now we evaluate the previous expressions at $Y$ satisfying (HP). Using the fact that $Y^TY$ is diagonal, we simplify:
\begin{equation}\label{hhh}
\begin{split}
\partial_{ab}h_{ij+m}(Y) &= \sum_{k = 1}^m\delta_{jk}^{ab}S_{ik} -\sum_{1\le l,k \le m}\left(S_{ib}S_{kl}y_{al}y_{jk} + y_{al}y_{jk}S_{il}S_{kb}\right) \\
&= \delta_{ja}\delta_{ib}S_{ii} -\sum_{1\le k \le m}\delta_{ib}S_{ii}S_{kk}y_{ak}y_{jk} - y_{ai}y_{jb}S_{ii}S_{bb}
\end{split}
\end{equation}
First, let $m \le n$. Then, using that for every $1 \le a,j,k \le n$ we have
\[
|S_{kk}y_{ak}y_{jk}| \le 1,
\]
we estimate
\begin{align*}
\mathcal{A}(Y)|\partial_{ab}h_{ij +m}(Y)| &\le \mathcal{A}(Y)|S_{ii}| + \mathcal{A}(Y)\sum_{1\le k \le m}|S_{ii}S_{kk}y_{ak}y_{jk}| +\mathcal{A}(Y)y_{ai}y_{jb}S_{ii}S_{bb}\\
& \lesssim \frac{\mathcal{A}(Y)}{1 + \|Y^i\|^2} + \frac{\mathcal{A}(Y)}{\sqrt{1 + \|Y^i\|^2}\sqrt{1 + \|Y^b\|^2}} \lesssim 1 + \|Y\|^{m - 1}.
\end{align*}
If $n < m$, let $J = \{j_1,\dots, j_n\}$ be the set of indices defined in (HP). If there exists $\ell$ such that $Z^{j_\ell} = 0$, then
\[
\mathcal{A}(Y) = \sqrt{\prod_{t \in J}(1 + \|Y^t\|^2)} = \sqrt{\prod_{t \in J\setminus \{j_k\}}(1 + \|Y^t\|^2)} \lesssim 1 + \|Y\|^{n - 1}
\]
and
\[
|\partial_{ab}h_{ij+m}(Y)| \le S_{ii} + \sum_{1\le k \le m}S_{ii}S_{kk}|y_{ak}y_{jk}| + |y_{ai}y_{jb}|S_{ii}S_{bb} \lesssim 1,
\]
therefore
\[
\mathcal{A}(Y)|\partial_{ab}h_{ij}(Y)| \lesssim 1 + \|Y\|^{n - 1}.
\]
Hence we are just left with the case $n < m$ and $Y^{j_\ell} \neq 0$ for every $1 \le \ell \le n$. If this is the case, (HP) implies that $y_{kj_\ell} = \delta_{k\ell}y_{\ell j_\ell}$, for $ 1 \le k \le n$, and $y_{kj} = 0$ if $j \notin J$ and $1 \le k \le n$. Therefore, recalling \eqref{hhh},
\[
\partial_{ab}h_{ij+m}(Y) = 
\begin{cases}
S_{ii} -S_{ii}S_{j_aj_a}y^2_{a j_a} - y_{ai}y_{jb}S_{ii}S_{bb}& \text{ if } j = a, i = b,\\
-y_{ai}y_{jb}S_{ii}S_{bb}& \text{ otherwise}.
\end{cases}
\] 
In the first case, if $j = a, i = b$, we have $$S_{j_aj_a} = \frac{1}{1 + \|Y^{j_a}\|^2} = \frac{1}{1 + y^2_{aj_a}},$$ hence
\[
1 - S_{j_aj_a}y^2_{aj_a} = 1 - \frac{y^2_{aj_a}}{1 + y^2_{aj_a}} = \frac{1}{1 + y^2_{aj_a}} = \frac{1}{1 + \|Y^{j_a}\|^2}
\]
that implies
$$\partial_{ab}h_{ij+m}(Y) =  \frac{1}{1 + \|Y^{j_a}\|^2} - \frac{y_{ai}y_{jb}}{(1 + \|Y^i\|^2)(1 + \|Y^b\|^2)},$$
and it is now easy to see that
\[
\mathcal{A}(Y)|\partial_{ab}h_{ij+m}(Y)| \lesssim 1 + \|Y\|^{n - 1}.
\]
Since if $j \neq a$ or $b \neq i$, $\partial_{ab}h_{ij+m}(Y) = -y_{ai}y_{jb}S_{ii}S_{bb}$, the same estimate follows. To finish the second case, we still need to show that
\[
\mathcal{A}(Y)|(Y^TDh_{ij+m}(Y))_{ab}| \lesssim 1 + \|Y\|^{\min\{m,n\} - 1}.
\]
To do so, we recall \eqref{secondsecond} to estimate
\begin{align*}
|(Y^TDh_{ij+m}(Y))_{ab}| &\le \sum_{k =1}^m S_{ib}|y_{jk}|S_{ka} +\sum_{k = 1}^m|y_{jk}|S_{kb}\delta_{ai} + \sum_{k = 1}^m|y_{jk}|S_{kb}S_{ai}.
\end{align*}
With similar computations to the one to prove \eqref{ff}, we estimate for $1\le i,b,a,k \le m, 1 \le j \le n$,
\[
\mathcal{A}(Y)S_{ib}|y_{jk}|S_{ka} \le
\begin{cases}
0 &\text{ if } Y^k = 0 \text{ or } k \neq a,\\
\sqrt{\prod_{l \neq k}(1 + \|Y^l\|^2)} &\text{ otherwise},
\end{cases}
\]
that implies $$\mathcal{A}(Y)S_{ib}|y_{jk}|S_{ka} \lesssim 1 + \|Y\|^{\min\{m,n\} - 1}.$$ Finally, since also for every $1\le j \le n$, $1\le k,b \le m$
\[
\mathcal{A}(Y)|y_{jk}|S_{kb} \le
\begin{cases}
0 &\text{ if } Y^k = 0 \text{ or } k \neq b,\\
\sqrt{\prod_{l \neq k}(1 + \|Y^l\|^2)} &\text{ otherwise},
\end{cases}
\]
we find
\[
\mathcal{A}(Y)|y_{jk}|S_{kb} \lesssim 1 + \|Y\|^{\min\{m,n\} - 1}, \forall 1\le k,b \le m, 1\le j \le n.
\]
This completes the proof of the second case.
\\
\\
\fbox{Third case, $m+1 \le i \le j \le m+n$:} As above we use $m+i$ and $m+j$ in place of $i$ and $j$. The indices $i$ and $j$ will then satisfy $1\leq i \leq j \leq n$ and we have $$h_{i+m,j + m}(X) = (XS(X)X^T)_{ij} = \sum_{1\le l,k \le m}x_{il}S_{lk}x_{jk}.$$ We compute the derivative using \eqref{partialfinal}:
\begin{align*}
\partial_{ab}h_{i+m, j+m}(X) &=  \sum_{1\le l,k \le m}\delta_{il}^{ab}S_{lk}x_{jk} +  \sum_{1\le l,k \le m}\delta_{jk}^{ab}S_{lk}x_{il} +  \sum_{1\le l,k \le m}x_{il}\partial_{ab}S_{lk}x_{jk}\\
&=\sum_{1\le k \le m}\delta_{ia}S_{bk}x_{jk} +  \sum_{1\le l \le m}\delta_{ja}S_{lb}x_{il}\\
& -\sum_{1 \le l,k,c \le m}S_{lb}S_{kc}x_{ac}x_{il}x_{jk} -\sum_{1 \le l,k,c\le m} x_{ac}S_{lc}S_{kb}x_{il}x_{jk}.
\end{align*}
Moreover,
\begin{equation}\label{LAST}
\begin{split}
(X^TDh_{i+m, j+m}(X))_{ab}= \sum_{1\le d \le n}x_{da}\partial_{db}h_{ij} &=\sum_{1\le d \le n, 1 \le k \le m}\delta_{id}x_{da}S_{bk}x_{jk} +  \sum_{1\le d \le n, 1 \le l \le m}\delta_{jd}S_{lb}x_{il}x_{da}\\
& -\sum_{1\le c,l,k \le m, 1 \le d \le n}S_{lb}S_{kc}x_{dc}x_{il}x_{jk}x_{da} \\
&-\sum_{1 \le c,l,k \le m, 1 \le d\le n} x_{da}x_{dc}S_{lc}S_{kb}x_{il}x_{jk}.
\end{split}
\end{equation}
By \eqref{def}, we have, for every $1 \le i,j \le m$
\[
\sum_{1\le k \le m,1\le d \le n}S_{ik}x_{dk}x_{dj} = \delta_{ij} - S_{ij}.
\]
Hence we can rewrite in \eqref{LAST}:
\begin{align*}
(X^TDh_{i+m, j+m}(X))_{ab}= \sum_{1\le d \le n}x_{da}\partial_{db}h_{ij} &=\sum_{1\le d \le n, 1 \le k \le m}\delta_{id}x_{da}S_{bk}x_{jk} +  \sum_{1\le d \le n, 1 \le l \le m}\delta_{jd}S_{lb}x_{il}x_{da}\\
& -\sum_{1 \le l,k \le m}S_{lb}x_{il}x_{jk}(\delta_{ka} - S_{ka}) -\sum_{1 \le l,k\le n} S_{kb}x_{il}x_{jk}(\delta_{la} - S_{la})\\
& =\sum_{1\le d \le n, 1 \le k \le m}\delta_{id}x_{da}S_{bk}x_{jk} +  \sum_{1\le d \le n, 1 \le l \le m}\delta_{jd}S_{lb}x_{il}x_{da}\\
& -\sum_{1 \le l,k \le m}S_{lb}x_{il}x_{jk}\delta_{ka} -\sum_{1 \le l,k\le m} S_{kb}x_{il}x_{jk}\delta_{la}\\
& +\sum_{1 \le l,k \le m}S_{lb}x_{il}x_{jk}S_{ka}+\sum_{1 \le l,k\le m} S_{kb}x_{il}x_{jk}S_{la}\\
& =\sum_{1\le k \le m}x_{ia}S_{bk}x_{jk} +  \sum_{1\le l \le m}S_{lb}x_{il}x_{ja}\\
& -\sum_{1 \le l \le m}S_{lb}x_{il}x_{ja} -\sum_{1 \le k\le m} S_{kb}x_{ia}x_{jk}\\
& +\sum_{1 \le l,k \le m}S_{lb}x_{il}x_{jk}S_{ka}+\sum_{1 \le l,k\le m} S_{kb}x_{il}x_{jk}S_{la}\\
&= \sum_{1 \le l,k \le m}S_{lb}x_{il}x_{jk}S_{ka}+\sum_{1 \le l,k\le m} S_{kb}x_{il}x_{jk}S_{la}.
\end{align*}
Consider once again $Y$ fulfilling (HP). Then:
\begin{align*}
\partial_{ab}h_{i+m, j+m}(Y)&=\sum_{1\le k \le m}\delta_{ia}S_{bk}y_{jk} +  \sum_{1\le l \le m}\delta_{ja}S_{lb}y_{il}\\
& -\sum_{1 \le l,k,c \le m}S_{lb}S_{kc}y_{ac}y_{il}y_{jk} -\sum_{1 \le l,k,c\le m} x_{ac}S_{lc}S_{kb}y_{il}y_{jk}.
\end{align*}
Since $Y^TY$ is diagonal, this expression simplifies as:
\begin{align*}
\partial_{ab}h_{i+m, j+m}(Y)&= \delta_{ia}S_{bb}y_{jb} + \delta_{ja}S_{bb}y_{ib}\\
& -\sum_{1 \le c \le m}S_{bb}S_{cc}y_{ac}y_{ib}y_{jc} -\sum_{1 \le c\le m} y_{ac}S_{cc}S_{bb}y_{ic}y_{jb}.
\end{align*}
For every $1 \le b \le m, 1 \le j \le n$,
\begin{equation}\label{thirdfirst}
\mathcal{A}(Y)S_{bb}|y_{jb}| \le
\begin{cases}
0& \text{ if } Y^b = 0,\\
\frac{\mathcal{A}(Y)}{\sqrt{1 + \|Y^b\|^2}}& \text{ otherwise}.
\end{cases}
\end{equation}
This yields
\[
\mathcal{A}(Y)S_{bb}|y_{jb}| \lesssim 1 + \|Y\|^{\min\{m,n\}- 1}.
\]
To prove that
\begin{equation}\label{thirdfirstfinal}
\mathcal{A}(Y)|\partial_{ab}h_{i+mj+m}(Y)| \lesssim 1 + \|Y\|^{\min\{m,n\} - 1},
\end{equation}
we still need to estimate terms of the form
\[
\mathcal{A}(Y)S_{bb}S_{cc}|y_{ac}y_{ib}y_{jc}|,
\]
for $1\le b,c \le m$, $1\le a,i,j \le n$. Anyway, observe that
\[
S_{cc}|y_{ac}||y_{jc}| \le 1, \forall 1\le c \le m, 1\le a,j \le n, 
\]
hence
\[
\mathcal{A}(Y)S_{bb}S_{cc}|y_{ac}y_{ib}y_{jc}| \le \mathcal{A}(Y)S_{bb}|y_{ib}|
\]
and we can therefore apply again estimate \eqref{thirdfirst} to deduce \eqref{thirdfirstfinal}. To finish the proof of this case and of the present Lemma, we still need to show that
\begin{equation}\label{thirdsecond}
|(Y^TDh_{i+m, j+m}(Y))_{ab}| \lesssim 1 + \|Y\|^{\min\{m,n\}- 1}. 
\end{equation}
To do so, recall that
\begin{align*}
(Y^TDh_{i+m, j+m}(Y))_{ab} &= \sum_{1 \le l,k \le m}S_{lb}y_{il}y_{jk}S_{ka}+\sum_{1 \le l,k\le m} S_{kb}y_{il}y_{jk}S_{la}\\
&= S_{bb}S_{aa}y_{ib}y_{ja} + S_{bb}S_{aa}y_{ia}y_{jb},
\end{align*}
but now, for every $1\le a,b \le m, 1 \le i,j \le n$,
\[
\mathcal{A}(Y)S_{bb}S_{aa}|y_{ib}y_{ja}| \le
\begin{cases}
0& \text{ if } Y^b = 0 \text{ or } Y^a = 0,\\
\frac{\mathcal{A}(Y)}{\sqrt{1 + \|Y^b\|^2}\sqrt{1 + \|Y^a\|^2}}& \text{ otherwise}.
\end{cases}
\]
The proof of \eqref{thirdsecond} is now analogous to the one of \eqref{thirdfirstfinal}.

\bibliographystyle{abbrv}
\bibliography{Inclusioni26}

\end{document}